\documentclass[10pt,twocolumn,twoside]{IEEEtran} 

\IEEEoverridecommandlockouts                              

\usepackage{graphics} 
\usepackage{epsfig} 
\usepackage{mathptmx} 
\usepackage{times} 
\usepackage{amsmath} 
\usepackage{amssymb}  
\usepackage{xcolor,soul}
\usepackage{comment}
\usepackage{relsize,exscale}
\usepackage{subfig}
\usepackage[export]{adjustbox}

\usepackage[americanresistors,americaninductors]{circuitikz}
\usetikzlibrary{calc}
\usepackage{tikz}
\usetikzlibrary{shapes,arrows}
\usetikzlibrary{decorations.pathmorphing}
\ctikzset{bipoles/thickness=0.7}
\ctikzset{bipoles/length=0.8cm}
\ctikzset{bipoles/diode/height=.375}
\ctikzset{bipoles/diode/width=.3}
\ctikzset{tripoles/thyristor/height=.8}
\ctikzset{tripoles/thyristor/width=1}
\ctikzset{bipoles/vsourceam/height/.initial=.7}
\ctikzset{bipoles/vsourceam/width/.initial=.7}
\tikzstyle{every node}=[font=\small]
\tikzstyle{every path}=[line width=0.7pt,line cap=round,line join=round]

\tikzstyle{block} = [draw, fill=blue!10, rectangle, minimum height=4em, minimum width=8em]
\tikzstyle{smallblock} = [draw, fill=blue!10, rectangle, minimum height=2em, minimum width=2em]
\tikzstyle{sum} = [draw, fill=blue!10, circle, node distance=1cm]
\tikzstyle{input} = [coordinate]
\tikzstyle{output} = [coordinate]
\tikzstyle{pinstyle} = [pin edge={to-,thin,black}]

\DeclareMathAlphabet{\mathcal}{OMS}{cmsy}{m}{n}
\SetMathAlphabet{\mathcal}{bold}{OMS}{cmsy}{b}{n}

\title{\LARGE \bf
Steady state characterization and frequency synchronization of a multi-converter power system  on high-order manifold}

\author{Taouba Jouini$^{1}$, Zhiyong Sun$^{2}$ 
\thanks{*This work has received funding from the European Research Council (ERC) under the European Union's Horizon 2020 research and innovation program (grant agreement No: 834142) and from the European Union's Horizon 2020 research and innovation program under grant agreement No: 691800 and ETH Z\"urich funds. }%
	\thanks{$^{1}$Taouba Jouini is with the Department of Automatic Control, LTH, Lund University,
		Ole Römers väg 1,  22363 Lund, Sweden. 
				$^{2}$Zhiyong Sun is with Department of Electrical Engineering,  Eindhoven University of Technology, the Netherlands.
		E-mails:
		\tt\small taouba.jouini@control.lth.se, z.sun@tue.nl.}}%

\graphicspath{{./fig/}}
\usepackage{epstopdf}

\usepackage{amsmath,amsthm,amssymb,mathrsfs,dsfont}
\usepackage{amsfonts}
\usepackage{siunitx}
\usepackage[colorinlistoftodos,prependcaption,textsize=tiny]{todonotes}
\usepackage{tikz}
\usetikzlibrary{matrix}
\usepackage{mathtools}

\newcommand\scalemath[2]{\scalebox{#1}{\mbox{\ensuremath{\displaystyle #2}}}}

\usepackage{cite}

\newcommand\oprocendsymbol{\hbox{$\blacksquare$}}

\newcommand\oprocend{\relax\ifmmode\else\unskip\hfill\fi\oprocendsymbol}

\newcommand{\real}[0]{\mathbb R}

\providecommand{\norm}[1]{\lVert#1\rVert}
\newcommand{\dd}[0]{\mathrm d}

\newtheorem{theorem}{Theorem}[section]
\newtheorem{lemma}[theorem]{Lemma}
\newtheorem{definition}[theorem]{Definition}

\newtheorem{proposition}[theorem]{Proposition}

\newtheorem{remark}[]{Remark}
\newtheorem{assumption}[]{Assumption}

\newcommand{\cir}[1]{{\textcircled{\raisebox{-0.9pt}{#1}}}}

\newcommand{\tb}[0]{\color{black}}

\usepackage[normalem]{ulem}
\newcommand\rout{\bgroup\markoverwith{\textcolor{red}{/}}\ULon} 

\usepackage[prependcaption,colorinlistoftodos]{todonotes}

\begin{document}

\maketitle
\thispagestyle{empty}
\pagestyle{empty}

\begin{abstract}
We investigate the stability properties of a multi-converter power system model, defined on a high-order manifold. For this, we identify its symmetry (i.e., rotational invariance) generated by a static angle shift and rotation of AC signals. We characterize the steady state set, primarily determined by the steady state angles and DC power input. Based {\tb on eigenvalue conditions} of its Jacobian matrix, {\tb we show asymptotic stability of the multi-converter system in a neighborhood of the synchronous steady state set by applying the center manifold theory. We guarantee the eigenvalue conditions via an explicit approach.}
Finally, we demonstrate our results based on a numerical example involving {\tb a network of} identical DC/AC converter systems.
\end{abstract}


\section{INTRODUCTION}
 Electricity production is one of the largest sources of greenhouse gas emissions in the world. Carbon-free electricity will be critical for keeping the average global temperature within the United Nation's target and avoiding the worst effects of climate change \cite{ieeespectrum}. 
 Prompted by these environmental concerns, the electrical grid has witnessed a major shift in power generation from conventional (coal, oil) into renewable (wind, solar) sources. The massive deployment of distributed, renewable generation had an elementary effect on its operation via power electronics converters interfacing the grid, deemed as game changers of the conventional analysis methods of power system stability and control.

{\em Literature review}: Modeling and stability analysis in power system networks is conducted as a matter of perspective from two different angles. First, network perspective suggests an {\em up to bottom} approach, where DC/AC converter dynamics are regarded as controllable voltage sources and voltage control is directly accessible. The most prominent example is droop control that leads to the study of second-order pendulum dynamics, emulating the Swing equation of synchronous machines \cite{kundur1994power}, which resembles the celebrated Kuramoto-oscillator \cite{dorfler2012synchronization}. The analogy drawn between the two models has motivated a vast body of literature that harness the results available for synchronization on the circle via Kuramoto oscillators to analyze the synchronization in power systems.
Second, a {\em bottom to up} approach {\tb derives} DC/AC converter {\tb models} from first-order principles, where the dynamics governing DC/AC converters are derived from the circuitry of DC-and AC-sides and the intermediate switching block, which can structurally match that of synchronous machines \cite{arghir2018grid}. Recently, the matching control  has been proposed {\tb in \cite{arghir2018grid}} as a promising control strategy, which achieves {\tb a} structural equivalence of the two models, and endows the closed-loop system with advantageous features (droop properties, power sharing, etc.). By augmenting the system dynamics with a virtual angle, the frequency is set to be proportional to DC-side voltage deviations, constituting a measure of power imbalance in the grid. This leads to the derivation of higher-order models that describe a network of coupled DC/AC converters on nonlinear manifolds with higher order than the circle.

Similar to the physical world, where the laws governing interactions in a set of particles are invariant with respect to static translations and rotations of the whole rigid body \cite{sarlette2009geometry}, power system trajectories are invariant under a static shift in their angles, or said to possess a {\em rotational invariance}. The symmetry of the vector field describing the power system dynamics, indicates the existence of a continuum of steady states for the multi-converter (with suitable control that induces/preserves angle symmetry) or multi-machine dynamics. In particular, the rotational invariance is the topological consequence of the absence of a reference frame or absolute angle in power systems and regarded thus far as a fundamental obstacle for defining suitable error coordinates. To alleviate this, a common approach in the literature is to perform transformations either resulting from projecting into the orthogonal complement, if the steady state {\tb set} is a linear subspace \cite{schiffer2019global}, or grounding a node \cite{tegling2015price}, where classical stability tools such as Lyapunov direct method can be deployed. 

{\tb To analyze power system stability, different conditions
have been proposed. In  \cite {arghir2018grid} and \cite{caliskan2014compositional}, sufficient stability conditions
are obtained for a single-machine/converter connected to a load. In \cite{dorfler2012synchronization}, a sufficient algebraic stability condition connects the synchronization of power systems
with network connectivity and power system parameters. Although
these conditions give qualitative insights into the sensitivities
influencing stability, they usually require strong and often
unrealistic assumptions. For example, the underlying models are of reduced order (mostly first or second order) \cite{johnson2013synchronization, dorfler2012synchronization}.
Reduced-order systems, where one infers stability of the whole
system from looking at only a subset of variables, are not a
truthful representation of the full-order system if important
assumptions are not met \cite{khalil2002nonlinear}. Some stability conditions are
valid only in radial networks \cite{schiffer2019global}. Moreover, explicit stability conditions require assumptions like strong mechanical or DC-side
damping \cite{dorfler2012synchronization}, whereas implicit conditions are based on semi-definite programming and thus not very insightful \cite{vu2015lyapunov}. 
}

{
{\em Contributions}:
In this work, we ask in essence two fundamental questions: 
   i) Under mild assumptions on input feasibility, how can we describe the behavior of the steady state trajectories of the nonlinear power system, in closed-loop with a suitable control, that induces/preserves the symmetry, e.g. the matching control \cite{arghir2018grid, 9304344}? 
   ii) Based on the properties of the steady state manifold, can we ensure local {\tb stability}, i.e., synchronization?
   
 To answer the first question, we study the behavior of the steady state manifold. For this, we derive a steady state map, which embeds {known} steady state angles {into} the DC power input as a function of the network topology and converter parameters. We show that the steady state angles fully describe the steady state behavior and determine all the other states. The steady state map depends on network topology, which is known to play a crucial role in the synchronization of power systems \cite{sarlette2009geometry, schiffer2019global}. 
     Since the vector field exhibits symmetry with respect to translation and rotation actions, i.e., under a shift in all angles and a rotation in all AC signals, the steady state manifold inherits the same property and every steady state trajectory is invariant under the same actions. This allows us to define set of equilibria that are generated under these actions.
    In this manner, we gain an overall perspective of the behavior characterizing the steady state set of the power system model. 
   
 We address the second question by showing {\tb asymptotic stability} of the nonlinear trajectories confined to a neighborhood to the steady state set of interest. 
{\tb For this,  we study the stability of the nonlinear dynamics  as a direct application of the center manifold theory to the multi-converter power system. We assume that the eigenvalues of the Jacobian evaluated at a point on the synchronous steady state set can be split into one zero mode and the remainder with real part confined to the left half-plane. Accordingly, we then decompose the nonlinear dynamics into two subsystems, whose dynamics are zero and Hurwitz, respectively. This allows to define a center manifold upon modal transformation, where we use the reduction principle \cite[p.195]{wiggins1990introduction} to deduce the stability of the trajectories of the multi-converter system from the dynamics evolving on the center manifold. The point-wise application of the center manifold theory allows to construct a neighborhood of the steady state set of interest and thereby showing its local asymptotic stability.
}

{\tb To satisfy the Jacobian eigenvalue condition in an explicit way}, we study the linearized system trajectories and pursue a parametric linear stability analysis approach at a frequency synchronous steady state. Towards this, we develop a novel stability analysis  for a class of partitioned linear systems characterized by a stable subsystem and a one-dimensional invariant subspace. We propose a {new} class of Lyapunov functions characterized by an  oblique projection onto the complement of the invariant subspace,  where the inner product is taken with respect to a matrix to be chosen as solution to Lyapunov and $\mathcal{H}_\infty$ Riccati equations.
Our approach has natural cross-links with analysis concepts for interconnected systems, e.g., small-gain theorem systems. For the multi-source power system model, we arrive at explicit stability conditions that depend only on the converter's parameters and steady-state values. In accordance with other works, our conditions require sufficient DC-side and AC-side damping.

}

{\em Paper organization} The paper unfurls as follows: Section \ref{sec: setup} presents the model setup based on a high-fidelity nonlinear power system model. Section \ref{sec: eq-manifold} studies the symmetry of its vector field,  characterizes {\tb the steady state set of interest}.
Section \ref{sec: local-asymptotic stability} studies local asymptotic {\tb stability} of the nonlinear power system model. {Section \ref{sec: sys-linea} shows asymptotic stability of the linearized power system dynamics and provides interpretations of our results.}
Finally, Section \ref{sec: sims} exemplifies our theory via simulations in two test cases.

{\em Notation}:
Define an undirected graph $\mathbb{G}=(\mathcal{V}, \mathcal{E})$, where $\mathcal{V}$ is the set of nodes with $\vert\mathcal{V}\vert =n$ and $\mathcal{E}\subseteq \mathcal{V}\times \mathcal{V}$ is the set of interconnected edges with $\vert \mathcal{E}\vert=m$. We assume that the topology specified by $\mathcal{E}$ is arbitrary and define the map $\mathcal{E} \to \mathcal{V}$, which associates each oriented edge $e_{ij}=(i,j) \in \mathcal{E}$ to an element from the subset $ \mathcal{I}=\{-1,0,1\}^{|\mathcal{V}|}$,  
resulting in the incidence matrix $\mathcal{B}\in\mathbb{R}^{n\times m}$.
We denote the identity matrix $I=\left[\begin{smallmatrix}
1& 0\\ 0& 1
\end{smallmatrix}\right]$, and $\mathbf{I}$ the identity matrix of {\tb suitable dimension} $p\in\mathbb{N}$, and $\mathbf{J}= \mathbf{I} \otimes \tb J_2$ with $J_2=\left[\begin{smallmatrix}
0& -1\\ 1& 0
\end{smallmatrix}\right]$.
We define the rotation matrix $R(\gamma)=\left[\begin{smallmatrix}
\cos(\gamma) & -\sin(\gamma) \\ \sin(\gamma) &\cos(\gamma) 
\end{smallmatrix}\right]$
and $\textbf{R}(\gamma)= \mathbf{I}\otimes R(\gamma)$. Let $\textrm{diag}(v)$ denote a diagonal matrix, whose diagonals are elements of the vector $v$ and $\mathrm{Rot}(\gamma)=\text{diag}(r(\gamma_{k})),\; k=1\dots n$, with $r(\gamma_k)=\begin{bmatrix} -\sin(\gamma_k) & \cos(\gamma_k)\end{bmatrix}^\top$. 
Let $\mathds{1}_n$ be the $n$-dimensional vector with all entries being one and $\mathbb{T}^n=\mathbb{S}^1\times \dots \times \mathbb{S}^1$ the $n$- dimensional torus. 
We denote by $d(\cdot, \cdot)$ be a distance metric. 
Given a set $\mathcal{A}\subseteq\mathds{R}^n$, then $d(z, \mathcal{A})=\inf\limits_{x\in \mathcal{A}} d(z,x)$ and $T_z\,\mathcal{A}$ is the tangent space of $\mathcal{A}$ at $z$.
Given a vector $v\in\real^n$, we denote {\tb by $v^{\perp}$ its orthogonal complement}, $v_k$ its $k$-th entry and $\oplus$ is the direct sum. For a matrix $A$, let $\norm{A}_2=\overline\sigma(A)$ denote its 2-norm and $\overline\sigma(A)$ denote its maximum singular value.
{\tb For convenience, we denote by $J_f(x^*)=\frac{\partial f(x)}{\partial x}\big\vert_{x=x^*}$ the Jacobian of $f$.}

\section{Power System model in closed-loop with matching control}
\label{sec: setup}
\subsection{Multi-source power system dynamics}
\label{subsec: model} 
\begin{figure}[h!]
\begin{minipage}{\columnwidth}
\centering
\begin{center}
\resizebox{8.5cm}{3cm}{
\begin{circuitikz}[american voltages]
\draw
(0,0) to [short, *-] (3,0)
(3.9,3) to [short, i>=$i_{net,k}$] (5,3)
(3.9,0) to [short, *-] (5,0)
(3,0) to (3.9,0)
(3,3) to (3.9,3)
(0.7,0) to [open, v^<=$v_{x,k}$] (0.7,3) 
(0,3) 
to [short,*-, i=$i_{k}$] (0.7,3) 
(0.6,3) to [R, l=$R$] (1.6,3) 
to [L, l=$L$] (2.6,3) 
(2.6,3) to (2.8,3)
to [short,*-] (2.8,2) 
(2.8,1) to [C, l=$C$] (2.8,2) 
(2.8,1) to [short] (2.8,0)
(3.9,0) to [R, l_=$G$] (3.9,3) 
(2.7,3) to (3,3)
(2.95,3) to [open, v^>=${v}_{k}$] (2.95,0); 
\draw
(-2,-0.5) to (-2,3.5)
(-0.8,1.5) node[pigbt] (pigbt){} 
(-2,3.5) to (0,3.5)
(0,3.5) to (0,-0.5)
(-2,-0.5) to (0,-0.5);
\draw
(-2,0) to [short, *-] (-4.3,0)
(-4.3,0) to [short, *-] (-4,0)
(-5.5,0) to (-4.3,0)
(-5.5,3) to (-4.3,3)
(-5.5,0) to [american current source, i>=$i^*_{dc,k}$] (-5.5,3) 
(-4.3,3) to [R,l=$G_{dc}$] (-4.3,0) 
(-3.1,3) to [C, l=$C_{dc}$] (-3.1,0) 
(-2,3) to [short, *-, i<=$i_{x,k}$] (-3,3)
(-4.3,3) to [short, *-] (-3,3)
(-3.1,3) to [short, *-] (-3,3)
(-5.2,3) to [open, v^>=${v}_{dc,k}$] (-5.2,0);
\end{circuitikz}
}
\end{center}
\caption{Circuit diagram of a balanced and averaged three-phase DC/AC converter with $i_{x,k}=\frac{\mu}{2}r^\top(\gamma_k)\,i_{k}$ and $v_{x,k}=\frac{\mu}{2} r(\gamma_k)\, v_{dc,k}$, \tb see e.g., \cite{wittig2009design}.} 
\label{fig:circuit diagram}
\end{minipage}
\end{figure}
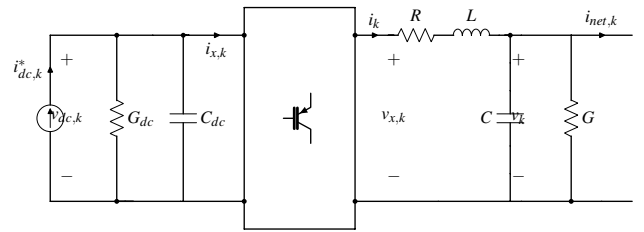

We start from the following general model describing the evolution of the dynamics of $n-$identical {three-phase balanced and averaged} DC/AC converters given in Figure \ref{fig:circuit diagram} in closed-loop with the matching control \cite{arghir2018grid}, a control strategy that renders the closed-loop DC/AC converter structurally similar to a synchronous machine, based on the concept of matching their dynamics; see Section \ref{subsec: contextualizing}. {The converter input $u_k$ is controlled as a sinusoid with constant magnitude $\mu\in]0,1[$ and frequency  $\dot \gamma\in\mathds{R}$ given by the DC voltage deviation.
\begin{subequations}
	\begin{align}
	\dot \gamma_k &= \eta (v_{dc,k}-v_{dc}^*)  \\
	u_k&= \mu \begin{bmatrix}
	-\sin(\gamma_k) \\
	\cos(\gamma_k)
	\end{bmatrix},\, k=1,\dots, n,
	\end{align}%
	\label{eq:matching-ctrl}%
\end{subequations}%
 where $\gamma_k\in\mathbb{S}^1$ is the virtual angle after a transformation into a  $dq$ frame, rotating at the nominal steady state frequency $\omega^*>0$, with angle $\theta_{dq}(t)=\int_0^t \omega^* \,d\tau$ (by the so-called Clark transformation, see \cite{kundur1994power}) and $\eta>0$ is a {\tb control} gain. 
}
 {The converters are interconnected with $m-$identical resistive and inductive lines.} The {closed-loop} converter dynamics are given by the following set of first-order differential equations in $dq$ frame. {\tb For the simplicity of notation, we will drop the subscript $dq$ from all the AC signals.}
\begin{align}
\begin{bmatrix} \dot {\gamma}_k  \\C_{dc} \dot v_{dc,k} \\ L\dot {i}_{k} \\ C\dot {v}_{k}   \end{bmatrix}=  \left[\begin{smallmatrix}
\eta (v_{dc,k}-v_{dc}^*) \\ -K_p (v_{dc,k}-v_{dc}^*)-\frac{\mu}{2} r(\gamma_{k})^{\top}i_{k}\\
-(R\, I+L\, \omega^*\, J )\, i_{k}+ \frac{\mu}{2} r(\gamma_{k}) v_{dc,k}-v_{k}\\
 -(G\, I+C\, \omega^*\,J)\, v_{k} +i_{k}-i_{net,k}
\end{smallmatrix}\right]+ \begin{bmatrix}
0 \\ i^*_{dc,k}\\ 0 \\ 0  
\end{bmatrix},
\label{eq: i-node}
\end{align}

Let the nominal frequency be given by $\omega^*\in\real$ and $v_{dc,k}\in\real$ denote the DC voltage across the DC capacitor with nominal value $v_{dc}^*$. The parameter $C_{dc}>0$ represents the DC capacitance and the  conductance $G_{dc}>0$, together with the proportional control gain $\hat K_p>0$, are represented by $K_p=G_{dc}+\hat K_p>0$. This results from designing a controllable current source $ i_{dc,k}=\hat K_p(v_{dc,k}-v_{dc}^*)+i^*_{dc,k}$, where we denote by $i^*_{dc,k}\in\real$ a constant current source representing DC side input to the converter. Let $i_{k} \in \mathbb{R}^2$ be the inductance current and $v_{k} \in  \mathbb{R}^2$ the output voltage. 
{The modulation amplitude $\mu$, feed-forward current $i_{dc}^*$ {and} the control gain $\widehat K_p$ are regarded as constants usually determined offline or in outer control loops. See \cite{arghir2018grid} for more details.
}
{\tb On the AC side,} the filter resistance and inductance are represented by $R>0$ and $L>0$ respectively. The capacitor $C>0$ is set in parallel with the load conductance $G>0$ to ground and connected to the network via the output current $i_{net,k}\in\real^{2}$.  

{Observe that the closed-loop DC/AC converter dynamics \eqref{eq: i-node} match one-to-one those of a synchronous machine with single-pole pair, non-salient rotor under constant excitation \cite{arghir2018grid}. Thus, all the results derived ahead can conceptually also be applied to synchronous machines (see also our comments in Section \ref{subsec: contextualizing}).
}
 
By lumping the states of $n-$identical converters and $m-$identical  lines and defining the impedance matrices $Z_R=R\; \mathbf{I}+L\,\omega^*\, \mathbf{J}, \, Z_C=G\;  \mathbf{I}+C\,\omega^*\,\mathbf{J},\, Z_\ell=R_\ell \; \mathbf{I}+L_\ell\omega^*\, \mathbf{J}$, we obtain the following power system model, 
\begin{align}
\begin{bmatrix} \dot {\gamma}  \\\dot v_{dc} \\ \dot {i}_{} \\ \dot {v}_{}  \\ \dot {i}_{\ell} \end{bmatrix}= K^{-1} \left[\begin{smallmatrix}
\eta (v_{dc}-v_{dc}^*\mathds{1}_n)\\ -K_p (v_{dc}-v^*_{dc}\mathds{1}_n)-\frac{1}{2}\mu \mathrm{Rot}(\gamma)^{\top}\, i_{}\\
-Z_R\,  i_{}+ \frac{1}{2}\mu \mathrm{Rot}(\gamma) \, v_{dc}-v_{}\\
-Z_C\, {{v}}_{}+i-\mathbf{B}\, i_\ell \\
-Z_\ell\, i_{\ell}+\mathbf{B}^\top\, v_{}
\end{smallmatrix}\right]+K^{-1} \begin{bmatrix}
0 \\ \textbf{u}\\ 0 \\ 0 \\ 0 
\end{bmatrix}\,,
\label{eq: multi-node}
\end{align}
where we define the angle vector $\mathbf{\gamma} =\begin{bmatrix}
\gamma_1, \dots, \gamma_n
\end{bmatrix}^{\top}\in \mathbb{T}^n$, with DC voltage vector ${v_{dc}} =\begin{bmatrix}
v_{dc,1}, \dots, v_{dc,n}
\end{bmatrix}^{\top}\in \mathbb{R}^n$,  the inductance current ${i} =\begin{bmatrix}
i_{1}^\top, \dots, i^\top_{n}
\end{bmatrix}^{\top}\in \mathbb{R}^{2n}$ and output capacitor voltage ${v} =\begin{bmatrix}
v_1^\top, \dots, v_n^\top
\end{bmatrix}^{\top}\in \mathbb{R}^{2n}$. 
The last equation in \eqref{eq: multi-node} describes the line dynamics and in particular, the evolution of the line current ${i_\ell}:=\begin{bmatrix}
i^\top_{\ell_1}, \dots, i^\top_{\ell_m}
\end{bmatrix}^\top \in \mathbb{R}^{2m}$, where $R_{\ell}>0$ is the line resistance, $L_{\ell}>0$ is the line inductance,
$\mathbf{B}=\mathcal{B}\otimes {I}$ and $K=\textrm{diag}\left(\mathbf{I}, C_{dc}\;\mathbf{I}, L\; \mathbf{I} ,C\; \mathbf{I},L_\ell\;\mathbf{I}\right)$. It is noteworthy that, $i_{net}=\mathbf{B}\, i_{\ell}$.
The multi-converter input is represented by $\mathbf{u}=\begin{bmatrix}
i^*_{dc,1}, \dots ,i^*_{dc,n}
\end{bmatrix}^\top\in\mathbb{R}^n$.

Let $N$ be the dimension of the state vector $z=\begin{bmatrix}
\gamma^\top & \tilde v_{dc} ^\top& x^\top
\end{bmatrix}^{\top}$, {where w}e define the relative DC voltage $\tilde {v}_{dc}= v_{dc}-v_{dc}^*\mathds{1}_n $, AC signals $x=\begin{bmatrix}
i^\top & v^\top & i^\top_\ell
\end{bmatrix}^{\top}$ and the input $u=\begin{bmatrix} 0^\top, \textbf{u}^\top, \dots, 0^\top\end{bmatrix}^\top\in~\real^N$ given by the vector \eqref{eq: multi-node}. 

By putting it all together, we arrive at the nonlinear power system dynamics compactly described by,
\begin{align}
\dot z={f}(z,  u),\;  
\label{eq: non-lin}
\end{align}
for all $z\in\mathcal{X}\subseteq\real^N$, where $\mathcal{X}$ is a smooth manifold and ${f}(z,u)$ denotes the vector field \eqref{eq: multi-node}.

\begin{remark}
\tb Without loss of generality, we assume that the DC/AC converters are identical and interconnected via identical RL lines which is a common assumption in the analysis of power system stability, see e.g., \cite{johnson2013synchronization,simpson2013synchronization}. Nonetheless, our analysis carries on to the more general heterogeneous setting, where the converters and the lines can be parameterized differently. See also Section \ref{subsec: contextualizing}.
\end{remark}

\section{Characterization of the steady state set}
\label{sec: eq-manifold}
 \subsection{Steady state map}
\begin{lemma}[Steady state map]
Consider the nonlinear power system model \eqref{eq: non-lin}. Given the steady state angles $\gamma^*$ {\tb satisfying $\dot\gamma^*=0$}. Then, {a feasible input \textbf{u}} {\tb is} given by, 
 \begin{align}
 \label{eq: angles}
   \mathbf{u}={ \xi}\, \text{Rot}(\gamma^*)^\top {Y}\; \text{Rot}(\gamma^*)\;\mathds{1}_n,
 \end{align}
  where {$\xi=\mu^2\frac{v_{dc}^*}{4}>0$} and ${Y}=({Z}_R+({Z}_C+\mathbf{B}\; Z_\ell^{-1} \;\mathbf{B}^\top)^{-1})^{-1}$.
\end{lemma}
 
 \begin{proof}
 To begin with, we solve for the steady state $z^*$ by setting \eqref{eq: non-lin} to zero. Note that $Z_C+\mathbf{B}\, Z_\ell^{-1}\mathbf{B}^\top$ and $Z_R+(Z_C+\mathbf{B} \, Z_\ell^{-1}\, \mathbf{B}^\top)^{-1}$ are non-singular {\tb matrices} due to the presence of the resistance $R>0$ and the load conductance $G>0$, where $\mathbf{B}\, Z_\ell^{-1}\mathbf{B}^\top$ is a weighted Laplacian matrix.
 The steady state of the lines is described by $i^*_\ell=Z^{-1}_\ell\, \mathbf{B}^\top v^*$, from which follow{\tb s} that $v^*=(Z_C+\mathbf{B} Z_\ell^{-1}\,\mathbf{B}^\top)^{-1}i^*$ for the output capacitor voltage {\tb at steady state}. The steady state inductance current is given by $i^*= \frac{1}{2} \mu {Y}\; \text{Rot}(\gamma^*)\,v^*_{dc}\mathds{1}_n$ and finally from $\frac{1}{2}\mu \text{Rot}^\top(\gamma^*) i^*=\mathbf{u}$, we deduce~\eqref{eq: angles}.
 \end{proof}
 
Notice that the matrix ${Y}\in\real^{2n\times 2n}$ in \eqref{eq: angles} has an {\em admittance-like} structure which is customary in the analysis of power system models and encodes in particular the admittance of the transmission lines according to the network topology given by the weighted network Laplacian $\mathbf{B}\; Z^{-1}_\ell \; \mathbf{B}^\top$, as well as the converter output filter parameters given by the impedance matrices $Z_R$ and $Z_C$. Once we solve for the steady state angles $\gamma^*\in\mathbb{T}
^n$, we recover the full steady state vector $z^*\in\mathcal{M}$. It is noteworthy that, each angle vector $\gamma^*$ determines a unique steady state $z^*\in\mathcal{M}$, which induces a steady state manifold $\mathcal{S}(z^*)$ as described in \eqref{eq: rot-symm}.

{Equation \eqref{eq: angles} can be understood as a steady state map (in the sense of \cite{isidori1990output}),
$$\mathcal{P}:  \mathds{T}^n\to  \real^n, \, \gamma^* \mapsto  \xi\,\text{Rot}(\gamma^*)^\top \mathcal{Y}\; \text{Rot}(\gamma^*)\,\mathds{1}_n ,$$ taking as argument the desired steady state angle $\gamma^*$ and mapping into the feasible input $\mathbf{u}$ pertaining to the set $\mathcal{U}$.
The steady state angles in \eqref{eq: angles} are obtained from solving an AC optimal power flow problem. Equation \eqref{eq: angles} is a power balance equation between electrical power $P^*_{e}= v_{dc}^*\,\xi Rot^\top(\gamma^*)\, i^*= v_{dc}^*\xi \,\text{Rot}(\gamma^*)^\top \mathcal{Y}\; \text{Rot}(\gamma^*)\;\mathds{1}_n$ and DC power given by $P^*_{m}=v_{dc}^*\mathbf{u}$. 
}
{In the sequel, we denote by $\mathcal{U}$ the set of all feasible inputs $\mathbf{u}$ given by \eqref{eq: angles}.
}
 \subsection{Steady state set}
Take $\textbf{u}\in\mathcal{U}$ and let $\mathcal{M}\subset\mathcal{X}$ be a non-empty steady-state manifold resulting from setting \eqref{eq: non-lin} to zero and given by,  
\begin{align}
\mathcal{M}=\left.\{z^*\in \mathcal{X} \, \vert \: {f}(z^*,u)=0 \right.\}.
\label{eq:ss-manifold}
\end{align}
{We are particularly interested in a synchronous steady-state with the following properties: 
\begin{itemize}
	\item The frequencies are synchronized at the nominal value $\omega^*$ mapped into a nominal DC voltage ${v_{dc}^*}\tb \geq 1$.
	\begin{align*}
	[\omega]&=\{\omega\in\mathds{R}_{\geq 0}^n|\,\omega=\omega^*\mathds{1}_n\}\,,\\
	[v_{dc}]&=\{ v_{dc}\in\mathds{R}_{\geq 0}^n|\, v_{dc}=v^*_{dc}\mathds{1}_n\}\,.
	\end{align*}
	\item 
	The angles are stationary 
	\[
	[\gamma]=\{\gamma\in \mathds{T}^n|\, \dot \gamma^*=0 \}\,.
	\]

	\item The inductor currents, capacitor voltage and line current are constant at steady state in a rotating frame
	\begin{align*}
	[i]&=\left\{i\in \mathbb{R}^{2n}|\dot i^*=0\right\}\,, [v]=\left\{v \in \mathbb{R}^{2n}| \dot v ^*=0\right\}\,,\\
	[i_{ \ell}]&=\left\{i_{ \ell}\in \mathbb{R}^{2m}| \dot i_{ \ell}^*=0\right\}. 
	\end{align*}
\end{itemize}

}
\subsection{Symmetry of the vector field}
\label{sec: symm}
Consider the nonlinear power system model  in \eqref{eq: non-lin}. For all $\theta\in\mathbb{S}^1$, it holds that,
\begin{align}
 f(\theta \, s_0+S(\theta)\, z,u)=f({\tb\mathcal{S}(z)},u)= S(\theta)\, f(z,u)\,,
\label{eq:invariance}
\end{align}
where we define the translation vector $s_0=\begin{bmatrix}
\mathds{1}^\top_n & 0^\top& 0^\top 
\end{bmatrix}^\top$, the matrix $S(\theta)=\left[\begin{smallmatrix}
\mathbf{I} & 0 & 0 \\ 0 & \mathbf{I} & 0\\ 0 & 0 & \mathbf{R}(\theta)
\end{smallmatrix}\right],\;$ and the set
\begin{align}
\label{eq: eq-class}
{\tb\mathcal{S}(z)}=\left\{\begin{bmatrix}
(\gamma +\theta \mathds{1}_n)^\top &  \tilde v_{dc}^\top & (\mathbf{R}(\theta)\,x)^\top \end{bmatrix}^\top, \, \theta\in\mathbb{S}^1 \right\}.
\end{align} 
The symmetry \eqref{eq:invariance} follows from observing that the rotation matrix $\mathbf{R}(\theta)$, commutes with the impedance matrices $Z_R,\,Z_C,\,Z_{\ell}$, the skew-symmetric matrix $\mathbf{J}$ and the incidence matrix $\mathbf{B}$. Notice that for $\theta=0$, it holds that  $S(z)=\{z\}$ and hence $z\in S(z)$.
In fact, the symmetry \eqref{eq:invariance} arises from the fact that the nonlinear power system model \eqref{eq: non-lin} has no absolute angle: A shift in all angles $\gamma\in\mathds{T}^n$, corresponding to a translation by $s_0$, induces a rotation in the angles of AC signals by $\mathbf{R}(\theta)$. Up to re-defining the $dq$ transformation angle to $\theta'_{dq}(t)=\theta_{dq}(t)+\theta$, the vector field \eqref{eq: non-lin} remains invariant under the translation $s_0$ and rotation action $S(\theta)$ in \eqref{eq:invariance}.


Consider the steady state manifold $\mathcal{M}$ described by \eqref{eq:ss-manifold}. Observe that a steady state $z^*\in \mathcal{M}$ pertains to a continuum of equilibria, as a consequence of the rotational symmetry \eqref{eq:invariance} and given by, 
 \begin{align}
 {\tb \mathcal{S}(z^*)}=\left\{\begin{bmatrix}
	 (\gamma^*+\theta\mathds{1}_n)^\top & \! 0^\top \!&  (\mathbf{R}(\theta)\,x^*)^\top
	\end{bmatrix}^\top,\,\theta\in\mathbb{S}^1 \right\}, 
	\label{eq: rot-symm}
   \end{align}
that is, for all $z^* \in \mathcal{M}$, it holds that $\mathcal{S}(z^*) \subset \mathcal{M}$. 




\section{Local synchronization of multi-converter power system}
{\tb
	\label{sec: local-asymptotic stability}
	
	In this section, we study local asymptotic stability of the steady state set $\mathcal{S}(z^*)$ {\tb in \eqref{eq: rot-symm}}, as an application of the center manifold theory  \cite[p.195]{wiggins1990introduction, carr2012applications}.

	\subsection{Preliminaries}
	We provide some background theory on center manifold theory \cite{carr2012applications}  which is our main tool for proving asymptotic stability.
	
	We review some key concepts from the center manifold theory. For this, consider a dynamical system in normal form,
	\begin{subequations}
		\label{eq: nrml-form}
		\begin{align}
		\dot y&= A_y y + f_1(y, \rho),\\  
		\dot \rho&= B_\rho \rho + f_2(y, \rho),
		\end{align} 
	\end{subequations}
where $A_y\in\real^{c\times c}$ has eigenvalues with zero real part and $B_\rho\in\real^{(n-c)\times (n-c)}$ has eigenvalues with negative real parts (or Hurwitz), and $f_1$ and $f_2$ are nonlinear functions with the following properties,
	\begin{align}
	& f_1(0,0)=0, J_{f_1}(0,0)=0, \\
	& f_2(0,0)=0, J_{f_2}(0,0)=0. 
	\end{align}
	An invariant manifold $\mathcal{W}^c$ is a center manifold of \eqref{eq: nrml-form}, if it can be locally represented as, 
	\begin{align}
	\label{eq: center-man}
	\mathcal{W}^c=\{(y,\rho) \in\mathcal{O}\vert \,\rho=h(y) \},
	\end{align}
	where $\mathcal{O}$ is a sufficiently small neighbourhood of the
	origin, $h(0)=0$ and $$J_h(0)=\frac{\dd\, h}{\dd y}{\bigg\vert_{y=0}}=0.$$
	
	It has been shown in \cite[Thm. 8.1]{khalil2002nonlinear} that a center manifold always exists and
	the dynamics of \eqref{eq: nrml-form} restricted to the center manifold are described by, 
	\begin{align}
	\label{eq: red-dyn}
	\dot\xi=A_y \xi +f_1(\xi, h(\xi)),    
	\end{align}
	for a sufficiently small $\xi\in\real^c$. Note that $\xi$ is a parametric representation of the dynamics along points on the center manifold $\mathcal{W}^c$ in \eqref{eq: center-man}.
	
	The stability of the system dynamics \eqref{eq: nrml-form} is analyzed from the dynamics on the center manifold using the reduction principle described in the following theorem.

	\begin{theorem}[\cite{wiggins1990introduction}, p.195]
		\label{thm: wiggins}
		If the origin is stable under \eqref{eq: red-dyn}, then the origin of \eqref{eq: nrml-form} is also stable. Moreover there exists a neighborhood $\mathcal{O}$ of the origin, such that for every $(y(0),\rho(0))\in\mathcal{O}$, there exists a solution $\xi(t)$ of \eqref{eq: red-dyn} and constants $c_1, c_2>0$ and $\gamma_1,\gamma_2>0$ such that, 
		\begin{align*}
		y(t)&=\xi(t)+r_1(t),\\
		\rho(t)&=h(\xi(t))+r_2(t),
		\end{align*}
		where $\norm{r_i(t)}<c_i\, e^{-\gamma_i\, t},\, i=1,2$.
	\end{theorem}
	
	Next, we provide background on set stability in the following definition.
	
	\begin{definition}[Set stability \cite{angeli2004almost}]
			\label{def: set-stb}
		A set $\mathcal{K}$ is called stable with respect to the dynamical system \eqref{eq: non-lin}, if for all $\epsilon>0$, there exists $\delta>0$, so that,
		\begin{align}
		\label{eq: stable}
		d(z_0, \mathcal{K})\leq \delta \implies d(z(t,z_0), \mathcal{K})<\epsilon,\quad \forall t\geq 0
		\end{align}
	\end{definition}
	A set $\mathcal{K}$ as in Definition \ref{def: set-stb} is  called  {\em asymptotically stable} with respect to a dynamical system \eqref{eq: non-lin}, if \eqref{eq: stable} holds and $$\lim_{t\to\infty} d(z(t,z_0),\mathcal{K})=0.$$
	
	\subsection{Local asymptotic stability}
	Next, we present our main result on local asymptotic stability of the set $\mathcal{S}(z^*)$ with respect to the multi-converter dynamics \eqref{eq: non-lin}. 
	
	The following assumption on the eigenvalues of the Jacobian of the multi-converter system \eqref{eq: non-lin} is crucial to derive our main result.
	\begin{assumption}
		\label{ass:eig-cdt}
	 Consider the linearized system described by the following equations,
		\begin{align}
		\label{eq: lin-sysmain}
		\delta {\dot z}&= J_f(z^*) \delta z.
		\end{align}
	Assume that $J_f(z^*)=\frac{\dd {f}}{\dd z}\left.\vert\right._{z=z^*}$ in  \eqref{eq: lin-sysmain}  representing the Jacobian of multi-converter system \eqref{eq: non-lin} linearized at $z^*\in\mathcal{M}$, has only one eigenvalue at zero and the real-parts of all other eigenvalues are in the open-left half plane.
	\end{assumption}
	
	\begin{remark}
	In Section \ref{sec: sys-linea}, we provide an approach on how to satisfy the eigenvalue condition in Assumption \ref{ass:eig-cdt} for the multi-converter system in an explicit  way.
	\end{remark}
	
	We now present our main result in the following theorem.
	
	\begin{theorem}[Local asymptotic stability]
		\label{thm: las}
		Consider the power system  dynamics in \eqref{eq: non-lin} under Assumption \ref{ass:eig-cdt} with a feasible input $\mathbf{u}\in\mathcal{U}$. Then,{\tiny } $\mathcal{S}(z^*)$ is locally asymptotically stable. Moreover, there exists a neighborhood $\mathcal{D}$ of $\mathcal{S}(z^*)$ such that for every $z(0)\in\mathcal{D}$, there exists a point $s\in\mathcal{S}(z^*)$, where 
		$$\lim_{t\to\infty} z(t)=s.$$
	\end{theorem}
	
	\begin{proof}
		To prove that $\mathcal{S}(z^*)$ is stable, we consider the system dynamics \eqref{eq: non-lin} under Assumption \ref{ass:eig-cdt}.
		
		Without loss of generality, assume $z^*=0$. From Assumption \ref{ass:eig-cdt}, we know there exists a transformation $T\in \real^{N\times N}$, such that $T \,J_f(0)\, T^{-1}$ is block diagonal, where $J_f(0)$ is given in \eqref{eq: lin-sys}, with a zero for the first component and a block $B$ that is Hurwitz. We rewrite the dynamics of \eqref{eq: non-lin} as
		\begin{align*}
		\dot z = J_f(0)\, z + (f(z,u)- J_f(0)\, z)
		\end{align*}
		where $z$ is near the origin $0\in\mathcal{S}(z^*)$. Next, by defining $(y,\rho)=T \,z$, we arrive at the following system in normal form
		\begin{subequations}
			\begin{align}
			\dot y&= f_1(y,\rho)\\
			\dot \rho&=B\, \rho+f_2(y,\rho),
			\end{align}
			\label{eq: sys-dyn}
		\end{subequations}
		where $f_1(0,0)=0, f_2(0,0)=0$ and $J_{f_1}(0,0)=J_{f_2}(0,0)=0$.
		
		Now, we show that, $$\mathcal{W}^c:=\{(y, \rho)\vert(\exists \, z\in\mathcal{S}(0)) \times (y, \rho)=T\, z\},$$
		is a center manifold for the system dynamics \eqref{eq: sys-dyn}. 
		
		First, $\mathcal{W}^c$ is invariant because it consists of
		equilibria of \eqref{eq: sys-dyn}. Second, $\mathcal{W}^c$ is tangent to the $y$-axis at $y=0$. To see this, define
		$$\tilde f(y,\rho):=f\left(T^{-1} \begin{bmatrix}
		y \\ \rho
		\end{bmatrix}\right)=f(z).$$
		Then $\mathcal{W}^c=\{(y,\rho)\vert \tilde f(y,\rho)=0 \}$. 
		The row vectors of the Jacobian  given by $$J_{\tilde f}(0,0)=\begin{bmatrix}
		\frac{\partial \tilde f_1(0,0)}{\partial y} &\frac{\partial \tilde f_1(0,0)}{\partial \rho}\\  \vdots & \vdots  \\ \frac{\partial\tilde f_N(0,0)}{\partial y}  & \frac{\partial \tilde f_N(0,0)}{\partial \rho}
		\end{bmatrix}=\frac{\dd f }{\dd z}\bigg\vert_{z=0}\, T^{-1}=J_f(0)\, T^{-1}$$ 
		span the normal space of $\mathcal{W}^c$ at $0$. Since the columns of $T^{-1}=(v(0),\dots )$ consist of the right eigenvectors of $J_f(0)$, 
		by means of $J_f(0)\,v(0)=0$,  $J_f(0)\, T^{-1}$ has a zero first column. This shows that $J_{\tilde f}(0,0)$ has its first entry (corresponding to $y-$component) equal to zero. 
		As a consequence, there exists a function $h(y)$ such that $h(0)=0$ and $\frac{\dd h}{\dd y}\vert_{y=0}=0$ in a neighborhood $\mathcal{W}_0$ of $0$, where $\mathcal{W}^c\cap \mathcal{W}_0 =\{(y,\rho)\vert \rho=h(y)\}$.
		
		It follows that the dynamics restricted to $\mathcal{W}_0$ are given by 
		$\dot\xi=0$ because $\mathcal{W}^c$ is an equilibrium manifold to \eqref{eq: sys-dyn} and thus $f_1(\xi,h(\xi))=0$. This shows that $\xi(t)=\xi(0)$.
		By applying Theorem \ref{thm: wiggins}, the solutions for $(y, \rho)$ starting in $\mathcal{W}_0$ are described by, 
		\begin{align*}
		y(t)&= \xi(t) + r_1(t),\\
		\rho(t)&= h(\xi(t)) + r_2(t),
		\end{align*}
		where $\norm{r_i(t)}<c_i e^{-\gamma_i t}, \, i=1,2$ for some constants $c_i,\gamma_i>0$. This implies that, 
		$$\lim_{t\to\infty} (y(t), \rho(t))=(\xi(0),h(\xi(0))), $$ and thus, $$\lim_{t\to\infty} z(t)= T^{-1} (\xi(0),h(\xi(0)))\in\mathcal{S}(0).$$
		
		This argument can be repeated for each point on $\mathcal{S}(0)$ to obtain a cover $\{\mathcal{W}_k\}$ of  $\mathcal{S}(0)$. Since $\mathcal{S}(0)$ is compact, we can construct a finite subcover to form a neighbourhood $\mathcal{D}=\bigcup_k \mathcal{W}_k$ of $\mathcal{S}(0)$. Local asymptotic stability of $\mathcal{S}(0)$ follows directly.
	\end{proof}
}
Note that our results conceptually apply to prove {\tb local asymptotic stability of a synchronous steady state set} with respect to trajectories of high-order dynamics of synchronous machines and find an estimate of their region of attraction, based on the structural similarities between synchronous machines and DC/AC converter in closed-loop with the matching control \cite{arghir2018grid}. Our {\tb local} analysis also paves the way for a global analysis of the stability of high-order multi-machine or multi-converter system with non-trivial conductance, which is an open problem in the power system community \cite{willems1976comments,chiang1989study}.

{ \section{\tb Sufficient conditions for the stability of the linearized  system}
 \label{sec: sys-linea}
{\tb This section derives sufficient conditions to satisfy the eigenvalue decomposition described in Assumption \ref{ass:eig-cdt} in an {\em explicit} way for a class of linear systems that applies to the stability of the linearized multi-converter system}. 
\subsection{Lyapunov stability of vector fields with symmetries}
{\tb In this section}, we develop a stability theory for a general class of linear systems enjoying some of the structural properties featured by the Jacobian matrix  in \eqref{eq: non-lin}. For this, we consider a class of partitioned linear systems  of the form
\begin{equation}
\label{eq:lin-sys-dq}
\dot x=\left[\begin{array}{c|c}
A_{11} & A_{12} \\ \hline A_{21} & A_{22}
\end{array}\right] x, 
\end{equation} 
where $x = [x_{1}^{\top}~x_{2}^{\top}]^{\top}$ denotes the partitioned state vector, and the block matrices $A_{11}, A_{12}, A_{21}, A_{22}$ are of appropriate dimensions.

In the following, we assume stability of the subsystem characterized by $A_{11}$ and the existence of a symmetry, i.e., an invariant zero eigenspace.
\begin{assumption}
In  \eqref{eq:lin-sys-dq}, the block {\tb diagonal} matrix $A_{11}$ is Hurwitz. 
\label{ass:a11-matrix}
\end{assumption}

\begin{assumption}\label{ass:v-vector}
	There exists a vector $p=[ p_1^{\top} ~ p_2^{\top}]^{\top}$, so that 
	\[
	A \cdot \text{span}\{p\}=0.
	\]
\vspace{-1.5em}
\end{assumption}

We are interested in asymptotic stability  of the subspace $\text{span}\{p\}$: all eigenvalues of $A$ have their real part in the open left-half plane except for one at zero, whose eigenspace is $\text{span\{p\}}$. Recall that the standard stability definitions and Lyapunov methods extend from stability of the origin to stability of closed and invariant sets when using the point-to-set-distance rather than merely the norm in the comparison functions; see e.g., \cite[Theorem 2.8]{lin1996smooth}. In our case, we seek a quadratic Lyapunov function that vanishes on $\text{span}\{p\}$, is positive elsewhere and whose derivative is decreasing everywhere outside $\text{span}\{v\}$.

We start by defining a Lyapunov function candidate,
\begin{align}
\label{eq:LF1}
V(x) = x^{\top} \,\left(P-\frac{P p p^{\top}P}{ p^{\top} P p}\right) x\,,
\end{align}
where $P$ is {\it a positive definite} matrix.
Our Lyapunov candidate construction is based on two {\it key} observations: 
\begin{itemize}
	\item First, the function ${V}(x)$ is defined via an oblique projection of the vector $x\in\real^n$ parallel to $\text{span}\{p\}$ onto $\{x \in\real^n \vert \,p^\top P\, x=0\}$. If $P=\mathbf{I}$, then $V$ is the {\em orthogonal projection onto $\mathrm{span}\{p\}^\perp$}. 
	Hence, $V(x)$ vanishes on $\text{span}\{p\}$ and is strictly positive definite elsewhere.
	\item Second, the positive definite matrix $P$ is {\it a degree of freedom} that can be specified later	to provide sufficient stability conditions. 

\end{itemize}

In standard Lyapunov analysis, one seeks a pair of matrices $(P,\mathcal{Q})$ with suitable positive (semi-) definiteness properties so that the  Lyapunov equation  $P\,A + A^{\top}P = -\mathcal{Q}$ is met. In the following, we apply a helpful twist and parameterize the $\mathcal{Q}$-matrix as a quadratic function $\mathcal Q(P)$ of $P$, which renders the  Lyapunov equation to an {$\mathcal{H}_\infty$} algebraic  Ricatti equation. We choose the following structure for the matrix $\mathcal{Q}(P)$,
\begin{align}
\label{eq:q-matrix}
\mathcal{Q}(P)=\begin{bmatrix}
\mathcal{Q}_1 & {H}^{\top}(P) \\ {H}(P) & {H}(P)\mathcal{Q}_1^{-1} {H}(P)^{\top}+\mathcal{Q}_2
\end{bmatrix},
\end{align}
where $\mathcal{Q}_1$ is a positive definite matrix, $\mathcal{Q}_2$ is a positive semi-definite matrix with respect to $\text{span}\{p_2\}$, $P$ is block-diagonal,
\begin{align}
\label{eq:matrix-P-}
P=\begin{bmatrix}
P_{1}  & 0 \\  0 &  P_{2}
\end{bmatrix}, 
\end{align}
with $P_{1}=P_{1}^{\top}>0$ and $P_{2}= P_{2}^{\top}>0$, i.e., the Lyapunov function is {\it separable}, and finally
${H}(P)= {A}^{\top}_{12}\, P_{1}+
P_{2}\, A_{21}$ is a shorthand. 

We need to introduce a third and final assumption.
\begin{assumption} 
\label{ass:hurwitz-cdt-}
Consider the matrix $
{F}=A_{22}+ A_{21}\mathcal{Q}_1^{-1} P_{1}A_{12}
$ and the transfer function,  $$\mathcal{G}=\mathcal{C}\, (s\,\mathbf{I}-F)^{-1}\, B,$$ with $ B=A_{21}\mathcal{Q}_1^{-1/2},\, \mathcal{C}=(A_{12}^\top P_1 \mathcal{Q}_1^{-1}P_1 A_{12}+\mathcal{Q}_2)^{1/2}$. Assume that ${F}$ is Hurwitz and that $\norm{\mathcal{G}}_\infty<1$.
\end{assumption}

Assumption~\ref{ass:hurwitz-cdt-} will guarantee suitable definiteness and decay properties of the Lyapunov function \eqref{eq:matrix-P-} under comparatively mild conditions discussed in Section~\ref{subsec: contextualizing}.

Assumptions \ref{ass:a11-matrix}, \ref{ass:v-vector}, and \ref{ass:hurwitz-cdt-} recover our requirement for positive definiteness of the matrix $P$  in \eqref{eq:matrix-P-} and semi-definitness (with respect to $\text{span}\{p\}$) of $\mathcal Q(P)$ in \eqref{eq:q-matrix} as shown in the following.
\begin{proposition} 
Under Assumptions \ref{ass:a11-matrix}, \ref{ass:v-vector} and \ref{ass:hurwitz-cdt-}, the matrix $P$ in \eqref{eq:matrix-P-} {exists}, is unique and positive definite. 
\label{prop:P-matrix-} 
\end{proposition}

\begin{proof}
By calculating $P\,A+{A}^{\top} P = -\mathcal Q(P)$, where $A$ is as in \eqref{eq:lin-sys-dq}, $P$ is as in \eqref{eq:matrix-P-}, and $\mathcal Q(P)$ is as in \eqref{eq:q-matrix}, we obtain  
	\begin{align*}
	\left[\begin{smallmatrix}
	P_{1} \, A_{11}+  A_{11}^{\top} P_{1} & {H}(P)^{\top} \\ {H}(P) &  P_{2}\,{A}_{22}+ {A}_{22}^{\top} P_{2}
	\end{smallmatrix}\right]
	= -
	\left[\begin{smallmatrix}
	\mathcal{Q}_1 & {H}(P)^{\top} \\ {H}(P) &  {H}(P)\mathcal{Q}_1^{-1}{H}(P)^{\top}+ \mathcal{Q}_2
	\end{smallmatrix}\right]\,,
	\end{align*}
	the block-diagonal terms of which are
	 \begin{enumerate}
	 	\item[\cir{\small 1}] $P_{1}\, A_{11}+ {A}_{11}^{\top} P_{1}=-\mathcal{Q}_1$,
	 	\item[\cir{\small 2}] $ P_{2}\,  {A}_{22}+ {A}_{22}^{\top} P_{2}=-{H}(P)\mathcal{Q}_1^{-1}{H}(P)^{\top}-\mathcal{Q}_2$, 
	 \end{enumerate}
	 {\tb where ${H}(P)= {A}^{\top}_{12}\, P_{1}+P_{2}\, A_{21}$.}
	Since $A_{11}$ is Hurwitz, there is a unique and positive definite matrix $P_{1}$ solving $\cir{\small 1}$.
  Moreover, specification $\cir{\small 2}$ is equivalent to solving for $P_{2}$ in the following {$\mathcal{H}_\infty$} {algebraic} Riccati equation:
 \begin{align*}
\!\!\!P_{2}\, {A_{21}} \mathcal{Q}_1^{-1} A_{21}^{\top}P_{2}+P_{2} {F}+{F}^{\top} P_{2}+A_{12}^{\top} P_{1}\mathcal{Q}_1^{-1} P_1 A_{12}+\mathcal{Q}_2=0, 
\end{align*} 
 where ${F}=A_{22}+ A_{21}\mathcal{Q}_1^{-1} P_{1}A_{12}$.
{\tb Under Assumption \ref{ass:hurwitz-cdt-}, the pair $(F,B)$ is stabilizable with $B=A_{21}\mathcal{Q}_1^{-1/2}$ and for $\norm{\mathcal{G}}_\infty<1$, Theorem 7.4 in \cite{scherer2001theory} implies that no eigenvalues of the Hamiltonian matrix $\mathcal{H}=\left[\begin{smallmatrix}
F & B\, B^\top \\ -\mathcal{C}^{\tb\top}\,\mathcal{C} & -F^\top
\end{smallmatrix}\right]$ are on the imaginary axis with $\mathcal{C}=(A_{12}^\top P_1 \mathcal{Q}_1^{-1}P_1 A_{12}+\mathcal{Q}_2)^{1/2}$. By Theorem 7.2 in \cite{scherer2001theory}, there exists a unique stabilizing solution $P_{2}$ to $\cir{\small 2}$. Define $E=A_{12}^{\top} P_{1}\mathcal{Q}_1^{-1} P_1 A_{12}+\mathcal{Q}_2+P_{2}\, {A_{21}} \mathcal{Q}_{1}^{-1} A_{21}^{\top}P_{2}\geq 0$. From $Ap=0$ follows that $A_{12} p_2=-A_{11}p_1\neq 0$ and since $\mathcal{Q}_2 p_2=0$, $\ker{\mathcal{Q}_2}\cap\ker{A_{12}}=\{0\}$. This shows that $E$ is non-singular and thus $E$ is positive definite.
Since $F$ is Hurwitz, by standard Lyapunov theory \cite{khalil2002nonlinear}, the Lyapunov equation $P_2\,F+F^\top P_2+E=0$ admits a positive definite solution $P_2$.}
\end{proof}

\begin{lemma}
Under Assumptions \ref{ass:a11-matrix}, \ref{ass:v-vector} and \ref{ass:hurwitz-cdt-}, the matrix $\mathcal{Q}(P)$ in \eqref{eq:q-matrix} is positive semi-definite. Additionally, $\ker({A})=\ker(\mathcal{Q}(P))=\text{span}\{p\}$.
	\label{lem: q-prop}
\end{lemma}
\begin{proof}
First, {\tb note that by Proposition \ref{prop:P-matrix-}, the matrix $P=P^\top>0$ and} observe that the matrix $\mathcal{Q}(P)$ in \eqref{eq:q-matrix} is symmetric and the upper left block $\mathcal{Q}_1>0$ is positive definite. By using the Schur complement and positive semi-definiteness of $\mathcal{Q}_2$, we obtain that $\mathcal{Q}(P)$ is positive semi-definite. 
	Second, by virtue of $p^{\top}\mathcal{Q}(P) p= p^{\top}(P\,A +A^{\top}P) p = 0$ due to Assumption~\ref{ass:v-vector}, it follows that $\text{span}\{p\}{\subseteq}\ker(\mathcal{Q}(P))$. 
	Third, consider a general vector $s=\left[\begin{smallmatrix}
	s_1^{\top}  s_2^{\top} 
	\end{smallmatrix}\right]^{\top}$, so that $\mathcal{Q}(P)s=0$. {\tb Given ${H}(P)= {A}^{\top}_{12}\, P_{1}+ P_{2}\, A_{21}$}, we obtain the algebraic equations $\mathcal{Q}_1s_1+{H}(P)^{\top} s_2=0,\, {H}(P) s_1+\left({H}(P){\mathcal{Q}_1^{-1}}{H}(P)^{\top} +\mathcal{Q}_2\right)s_2=0$.
	One deduces that $\mathcal{Q}_2 s_2=0$ and thus $s_2\in \text{span}\{p_2\}$. The latter implies $s_1\in-\mathcal{Q}_1^{-1}{H}(P)^{\top} \text{span}\{p_2\} \in\text{span}\{
	p_1\}$ because $\mathcal{Q}(P)\text{span}\{p\}=0$. 
	Thus, it follows that $s\in\text{span}\{ \left[\begin{smallmatrix}
	s_1^\top & s_2^\top
	\end{smallmatrix}\right]^\top \} =\text{span}\{p\}$ and we deduce that   $\ker(\mathcal{Q}(P)) = \text{span}\{p\}$. 
	Fourth and finally, for the sake of contradiction, take a vector $\tilde v \notin \text{span}\{p\}\,,$ so that $\tilde v\in \ker ({A}) \Rightarrow \tilde v^{\top}\left({A}^{\top} P+P\,{A}\right)\tilde v=0 \Rightarrow \tilde v^{\top}\mathcal{Q}(P)\tilde v=0 \Rightarrow \tilde v \in\ker (\mathcal{Q}(P))$.
	This is a contradiction to $\ker (\mathcal{Q}(P))=\text{span}\{p\}$. Hence, we conclude that $\ker({A})=\ker (\mathcal{Q}(P))=\text{span}\{p\}$.
\end{proof}
Next, we provide the main result of this section.
\begin{lemma}
	Consider the linear system \eqref{eq:lin-sys-dq}. Under Assumptions \ref{ass:a11-matrix}, \ref{ass:v-vector} and \ref{ass:hurwitz-cdt-}, $\text{span}\{p\}$ is an {asymptotically} stable subspace of $A$.
	\label{lem: las}
\end{lemma}
\begin{proof}
Consider the function $V(x)$ in \eqref{eq:LF1}. The matrix $P$ in \eqref{eq:matrix-P-} is positive definite by Proposition \ref{prop:P-matrix-}.
By taking $y= P^{1/2} x$ and $w=P^{1/2} p$, the function ${V}(x)$ can be rewritten as 
${V}(y)
= y^{\top}\left(\textbf{I}-\frac{w w^{\top}}{w^{\top} w}\right)y=y^{\top} \Pi_{w} y$.
The matrix $\Pi_{w}=\textbf{I}-\frac{w w^{\top}}{w^{\top} w}$ is a projection matrix into the orthogonal complement of $\text{span}(w)$, and is hence positive semi-definite with one-dimensional nullspace corresponding to $P^{1/2}\text{span}\{p\}$.  It follows that the function ${V}(x)$ is positive definite {\tb for all $x\in\text{span}\{p\}^{\perp}$}. 
By means of ${A} p=\mathcal{Q}(P)p=0$ in $p^{\top}P\,{A}=p^{\top}(\mathcal{Q}{\tb(P)}-{A}^{\top}P)=0$, we obtain $ \dot V(x) = -x^{\top}\, \mathcal{Q}(P) \, x.$
By Lemma \ref{lem: q-prop}, it holds that $\dot V(x)$ is negative definite {\tb for all $x\in\text{span}\{p\}^{\perp}$}. We apply Lyapunov's method  and Theorem 2.8 in \cite{lin1996smooth} to conclude that $\text{span}\{p\}$ is asymptotically stable. 
\end{proof}

\subsection{Stability of the linearized multi-DC/AC converter}
\label{sec:localstabilitydcac}
 Our next analysis takes under the loop the behavior of linearized trajectories described by the Jacobian of \eqref{eq: non-lin} at $z=z^*$. For this, we consider the linearized system described by the following equations,
\begin{align}
\label{eq: lin-sys}
\delta {\dot z}&=K^{-1} \left[\arraycolsep=0.8pt\def\arraystretch{1}
\scalemath{0.7}{ 
\begin{array}{cc|ccc}
0 & \eta \textbf{I}   & 0 & 0 & 0 \\   -\nabla^2 U(\gamma^*) & -  K_p \textbf{I} & -  \Lambda(\gamma^*)^{\top} & 0 & 0 \\ \hline   \Xi(\gamma^*) &   \Lambda(\gamma^*) & - {Z}_R & -  \textbf{I} & 0 \\ 0 & 0 &   \textbf{I} & - {Z}_C & - \mathbf{B} \\ 0 & 0 & 0&  \mathbf{B}^{\top} &  - {Z}_{\ell}
\end{array}}\right] \delta z=\left[\begin{array}{c|c}
A_{11} & A_{12} \\
\hline
A_{21} & A_{22}
\end{array} 
\right] \delta z.
\end{align}
The system matrix is the Jacobian $J_f(z^*)=\frac{\dd {f}}{\dd z}\left.\vert\right._{z=z^*}$,  $\delta z=\begin{bmatrix}\delta z^\top_1 & \delta z^\top_2\end{bmatrix}^\top \in T_{z^*}\mathcal{M}$, corresponding to the partition $\delta z_1=\begin{bmatrix} \delta \gamma^\top & \delta v_{dc}^\top \end{bmatrix}^\top \in \real^{2n},\, \delta z_2\in \real^{6n}$.
The matrices in \eqref{eq: lin-sys} are given by, 
\begin{align*}
   \nabla^2 U(\gamma^*)&= \frac{ v^*_{dc}}{4} \mu^2\text{diag}(\text{Rot}^\top(\gamma^*)\; \textbf{J}^\top  \mathcal{Y}\; \text{Rot}(\gamma^*)\;\mathds{1}_n ),\\
   &=\frac{1}{2}\mu\text{diag}((\mathbf{J} \text{Rot}(\gamma^*))^\top i^*),
   \\
   \Xi(\gamma^*)&=\frac{1}{2}\mu \textbf{J}\,\text{Rot}(\gamma^*), \\ 
   \Lambda(\gamma^*)&=\frac{1}{2}\mu v_{dc}^*\, \text{Rot}(\gamma^*),  
  \end{align*}
  where we consider the smooth potential function, $$U: \mathbb{T}^n \to \mathbb{R}, \, \gamma\mapsto -\xi\;\mathds{1}^\top_n\; \text{Rot}^\top(\gamma) \,\mathbf{J}^\top \mathcal{Y}\; \text{Rot}(\gamma^*)\;\mathds{1}_n.$$
  
Note that the Jacobian $J_f(z^*)$ has one-dimensional zero eigenspace denoted by,  
 $$\mathrm{span}\{v(z^*)\}=\mathrm{span}\{\left[ \begin{array}{c c c} 
\mathds{1}_n {^\top} & 0^\top & (\mathbf{J} \, x^*)^\top \end{array}\right]^\top\} \subset T_{z^*}\mathcal{M},$$ with $\mathbf{J} \, x^*=\begin{bmatrix}(\mathbf{J}\,i^*){^\top} & (\mathbf{J}v_{}^*){^\top} & (\mathbf{J}\, i_{\ell}^*){^\top}\end{bmatrix}^\top$. In particular, we can establish a formal link between the linear subspace $\mathrm{span}\{v(z^*)\}$ and the steady state set $\mathcal{S}(z^*)$ in \eqref{eq: rot-symm} as follows.  For all $\theta\in\mathbb{S}^1$, 
\begin{align*}
\mathcal{S}(z^*)&=z^*+\int_{0}^\theta \;v(z^*) \; \mathrm{d}s,  =z^*+{\mathlarger{\mathlarger{\int}}_{0}^\theta \begin{bmatrix}
\mathds{1}_n \\0 \\ \mathbf{J}\, \mathbf{R}(s)\, x^* 
\end{bmatrix} \; \mathrm{d}s}, \,
\end{align*}
which follows from \eqref{eq: rot-symm}. In fact, $v(z^*)$ is the tangent vector of $\mathcal{S}(z^*)$ in the $\theta$- direction and lies on the tangent space $T_{z^{*}}\mathcal{M}$. Hence, $\mathcal{S}(z^*)$ is the angle integral curve of $\text{span}\{v(z^*)\}$. 

It can be deduced from \eqref{eq:invariance}, that by expanding Taylor series around $(\theta',z^*),\, \theta'\in\mathbb{S}^1,\ z^*\in\mathcal{M}$ of left and right terms in \eqref{eq:invariance} and comparing the terms of their first derivatives with respect to $\theta$, we recover $J_f(z^*)\, v(z^*)=0$, as follows, 
\begin{align*}
\left.\frac{\dd f(z)}{\dd z}\right \vert_{z=z^*} \left(\left.\frac{\dd S(\theta)}{\dd \theta}\right \vert_{\theta=\theta'} \; z^*+s_0\right) \; (\theta-\theta') = 0,
\end{align*}
where $\frac{\dd f(z)}{\dd z}\ \vert_{z=z^*}=J_f(z^*)$, $\frac{\dd S}{\dd \theta}\vert_{\theta=\theta^{' }} z^*+s_0=v(z^*)$ (by the definition of the set \eqref{eq: rot-symm}), {\tb $S(\theta)=\left[\begin{smallmatrix}
\mathbf{I} & 0 & 0 \\ 0 & \mathbf{I} & 0\\ 0 & 0 & \mathbf{R}(\theta)
\end{smallmatrix}\right],\;$ and $s_0=\begin{bmatrix}
\mathds{1}^\top_n & 0^\top& 0^\top 
\end{bmatrix}^\top$.} 

Next, we consider the linearized model \eqref{eq: lin-sys} and identify the matrices
\begin{align*}
A_{11}&=\left[\begin{smallmatrix}
0 & \eta \textbf{I} \\   -C_{dc}^{-1} \nabla^2 U(\gamma^*) & -C_{dc}^{-1}  K_p \textbf{I}
\end{smallmatrix}\right],\, A_{12}=\left[\begin{smallmatrix}
0 & 0 & 0 \\   -C_{dc}^{-1}  \Lambda(\gamma^*)^{\top}  & 0 & 0
\end{smallmatrix}\right],\\
A_{21} &=\left[\begin{smallmatrix}
 L^{-1} \Xi(\gamma^*) & L^{-1}  \Lambda(\gamma^*)  \\   0 & 0 \\ 0 & 0
\end{smallmatrix}\right],\, A_{22}=\left[\begin{smallmatrix}
-L^{-1} \mathbf{Z}_R & -L^{-1}  \textbf{I} & 0 \\   C^{-1}  \textbf{I} & -C^{-1} \mathbf{Z}_V & -C^{-1} \mathbf{B} \\  0& L_{ \ell}^{-1} \mathbf{B}^{\top} &  -L_{ \ell}^{-1} \mathbf{Z}_{ \ell}
\end{smallmatrix}\right].
\end{align*}

Define the Lyapunov function $V(z)$ as in \eqref{eq:LF1} with $p = v(z^{*})=[v^{*\top}_1, v^{*\top}_2]^\top$.
Hence, $V(z)$ is positive semi-definite with respect to $\text{span}\{v(z^*)\}$. 
Next, we fix the matrix $\mathcal{Q}(P)$ given by \eqref{eq:q-matrix}, where we set $\mathcal{Q}_1=\textbf{I}$,  $\mathcal{Q}_2=\textbf{I}-{v_2^*v_2^{*\top}/{v_2^{*\top}v_2^*}}$ 
and search for the corresponding matrix $P$ so that, $$P\,J_f(z^*)+J_f(z^*)^{\top} P=-\mathcal{Q}(P).$$

Analogous to \eqref{eq:matrix-P-}, we choose the block diagonal matrix 
\begin{align}
\label{eq:matrix-P}
P=\left[\begin{array}{cc|c}
P_{11} & P_{12} & 0 \\ P_{12} & P_{22} & 0 \\ \hline 0 & 0 &  P_{33}
\end{array}\right]= \left[\begin{array}{c|c}
P_1 & 0 \\ \hline 0  &  P_2
\end{array}\right],
\end{align}
where $P_{11}, P_{12}$ and $P_{22}$ are matrices of appropriate dimensions.
Notice that the chosen structure of $P_1$ and the zeros in the off-diagonals in $P$ originate from the physical intuition of the tight coupling between the angle of the converter and its corresponding DC voltage (proportional to the AC frequency), as enabled by the matching control \eqref{eq:matching-ctrl}. The same type of coupling comes into play in synchronous machines between the rotor angle and its frequency, due to the presence of the electrical power in the swing equation \cite{kundur1994power}.
The $4\times 4$ matrix $P_2$ is dense with off-diagonals coupling at each phase, the inductance current of one converter with the others.

In the sequel, we show that this structure allows for sufficient stability conditions.
 \begin{assumption}[Parametric synchronization conditions]
 \label{ass:linear-stab}
{\tb Consider $P_{x,k}=\frac{1}{2}v_{dc}^* \mu r^\top(\gamma^*_k) i_k^*>0$, $Q_{x,k}=\frac{1}{2}v_{dc}^* \mu r^\top(\gamma^*_k) {J}^\top i_k^*>~0$} and the matrix $F$ as given in Assumption \ref{ass:hurwitz-cdt-}. Assume the following condition is satisfied,
 \begin{align}
\label{eq: ac-side}
\tb\cos(\phi_k)<\sqrt{1-\frac{\alpha^2}{P_{x,k}^2+\alpha^2}},
 \end{align}
{\tb where $\cos(\phi_k)=\frac{P_{x,k}}{\sqrt{Q_{x,k}^2+P_{x,k}^2}}\in[0,1[$ is the power factor of $k-$th converter, $\alpha=\max \left\{\frac{\mu^{2} v_{dc}^{*2}}{16\, R}, \frac{\mu\, v_{dc}^{*2}}{4\, \sqrt{Y^{-2}-1}}\right\}, \, Y=\frac{1}{2}\mu v^*_{dc} L^{-1} \sup\limits_{\tb\zeta} \norm{(j{\tb\zeta} \mathbf{I} -F)^{-1}}_2$} with the condition $Y<1$.

Additionally, assume that, 
\begin{align}
\label{eq: dc-side}
 \!\!\!\frac{\frac{\tb\mu}{\tb2}(1+\eta\, C_{dc}v_{dc}^{*} Q_{x,k}^{-1})} {\sqrt{\frac{4}{v_{dc}^{*2}}(Y^{-2}-1)-\frac{1}{\tb 4}{\tb \mu^2} v_{dc}^{*2}\,Q^{-2}_{x,k}} }< K_p.
\end{align}
\end{assumption}	
Next, we provide the main result of this section.
 \begin{lemma} 
 \label{thm: stab-lin-sys}
Consider the linearized closed-loop multi-converter model \eqref{eq: lin-sys}.
Under Assumption \ref{ass:linear-stab}, the subspace $\text{span}\{v(z^*)\}$ is asymptotically stable.
 \end{lemma}
 \begin{proof}
 	Since $v(z^*) \in \ker(J_f(z^*))$,  Assumption \ref{ass:v-vector} is satisfied.
 	If Condition \eqref{eq: ac-side} is true, then $r^\top(\gamma^*_k) J^\top i_k^* >0$, for all $k=1,\dots n$, the sub-matrix $A_{11}$ is Hurwitz and hence Assumption \ref{ass:a11-matrix} is also valid.
 	
 	Next, we verify  Assumption \ref{ass:hurwitz-cdt-}.
 	First, the matrix $P_1\tb=~\begin{bmatrix}
 	    P_{11} & P_{12} \\ P_{21} & P_{22}
 	\end{bmatrix}$ can be identified from specification $\small \cir{1}$ with $\mathcal{Q}_1=\textbf{I}$ by {\tb the following expressions,}
 	{\begin{subequations}
 			\begin{align*}
 			P_{11}&= \frac{1}{\eta} \left[\frac{1}{2}  K_p (\nabla^2 U(\gamma^*))^{-1} +\frac{\nabla^2 U(\gamma^*)}{2 K_p} (\textbf{I}+\eta C_{dc} (\nabla^2 U(\gamma^*))^{-1}\, )\right],\\
 			P_{12}&=P_{12}^{\top}= \frac{1}{2} (\nabla^2 U(\gamma^*))^{-1}{C_{dc}}, \\	
 			P_{22}&= \frac{C_{dc}}{2 K_p}  \left(\textbf{I}+\eta C_{dc}\, (\nabla^2 U(\gamma^*))^{-1}\right).
 			\end{align*}
 	\end{subequations}}
 	The feasibility of specification $\small \cir{2}$ with the positive semi-definite matrix $\mathcal{Q}_2=\textbf{I}-\frac{v_2^{*}v^{*\top}_2}{v^{*\top}_2v_2^{*}}$ is given by
 	{\small
 	\begin{align}
 	\label{eq: are}
 	& P_{2}\, {A_{21}A_{21}}^{\top} P_{2}+P_{2} {F}+{F}^{\top} P_{2}+ {NN}^{\top}+\mathcal{Q}_2=0\,, 
 	\end{align}
 	}%
 	where $F=A_{22}+A_{21} P_1 A_{12}$ and ${N}=A_{12}^{\top} P_1$.
 
 	If {Assumption 3 is satisfied}, then there exists a positive {\tb definite matrix} $P_2$ that satisfies the $\mathcal{H}_\infty-$ARE in \eqref{eq: are}. 
 	
 	Next, we find sufficient conditions, for which ${F}$ satisfies the Lyapunov equation ${P}_F {F}+{F}^\top{P}_F=-\mathcal{Q}_F$.  We choose $P_{F}$ and $\mathcal{Q}_{F}$ to be block-diagonal matrices 
 	$P_F=\left[\begin{smallmatrix}
 	L & 0 & 0 \\ 0 &C & 0 \\ 0 & 0 & L_{ \ell}
 	\end{smallmatrix}\right]
 	\;,\;
 	\mathcal{Q}_F=\left[\begin{smallmatrix}
 	\Gamma & 0 & 0 \\ 0 & 2G\,\textbf{I} & 0 \\ 0 & 0 & 2 R_{ \ell} \textbf{I}
 	\end{smallmatrix}\right]\,,
 	$
	with $\Gamma=2R\,\textbf{I}+C_{dc}^{-1}\left(\Xi(\gamma^*)P_{12} \Lambda(\gamma^*)^{\top}+\Lambda(\gamma^*) P_{12}\Xi(\gamma^*)^{\top}\right)+2 C_{dc}^{-1} \left(\Lambda(\gamma^*) P_{22} \Lambda(\gamma^*)^{\top}\right)$ being itself block-diagonal. Aside from $\Gamma$, all diagonal blocks of $P_{F}$ and $Q_{F}$ are positive definite. We evaluate the block-diagonal matrix $\Gamma$ for positive definiteness by exploring its two-by-two block diagonals, where trace and determinant of each block are positive under  $Q_{x,k}^*=\frac{1}{2}v_{dc}^* \mu( r(\gamma^*_k))^\top {J}^\top i_k^*>\frac{\mu^{2} v_{dc}^{*2}}{16\, R}$,


Furthermore, we impose the condition $\norm{\mathcal{G}}_{\infty}<1$, by equivalently setting $\sup\limits_{\tb\zeta\in\real}\norm{\mathcal{C}\,(j {\tb\zeta} \mathbf{I}-F)^{-1} B}_2<1$, where $ \mathcal{C}={\left(A_{12}^\top P^\top_1 P_1 A_{12}+\mathbf{I}-\frac{(\mathbf{J} x^*)(\mathbf{J} x^*)^\top}{(\mathbf{J} x^*)^\top(\mathbf{J} x^*)}\right)^{1/2}},\, B=A_{21}$. 
	It is sufficient to consider $\norm{\mathcal{C}}^2_2< (\sup\limits_{\zeta\in\real}\norm{(j {\tb\zeta} \mathbf{I}-F)^{-1}}_2 \norm{B}_2)^{-2}$. Using the triangle inequality for the 2-norm, it holds that $\norm{\mathcal{C}}^2_2\leq \norm{A_{12}^\top P^\top_1 P_1 A_{12}}_2+\norm{\mathcal{Q}_2}_2$. Since $\norm{\mathcal{Q}_2}_2=1$, we consider instead,
	$$\norm{A_{12}^\top P^\top_1 P_1 A_{12}}_2\leq  (\sup\limits_{\tb\zeta}\norm{(j{\tb\zeta} I-F)^{-1}}_2 \norm{B}_2)^{-2}-1.$$
	From $\norm{B}^2_2=\norm{A_{21}}^2_2=L^{-2}\norm{\begin{bmatrix}
	\Xi^\top \Xi & \Xi^\top \Lambda \\  \Lambda^\top\Xi & \Lambda^\top \Lambda 
	\end{bmatrix}}_2=L^{-2}\norm{\begin{bmatrix}
	\Xi^\top \Xi & 0  \\  0 & \Lambda^\top \Lambda
	\end{bmatrix}}_2=L^{-2} (\frac{1}{2}\mu v^*_{dc})^2$ because $v_{dc}^*\geq1$.
	
{\tb Define $Y=\frac{1}{2}L^{-1}\mu v^*_{dc}\sup\limits_{\tb\zeta} \norm{({j\,\tb\zeta} \mathbf{I} -F)^{-1}}_2$. 
For $Y<1$,} straightforward calculations show that $\norm{A_{12}^\top P^\top_1 P_1 A_{12}}_2=\overline\sigma(C_{dc}^{-2}\Lambda(\gamma^*) (P^2_{12}+P^2_{22})\Lambda(\gamma^*)^\top)=\max\limits_{k=1,\dots, n} d_k\, \overline\sigma(r(\gamma_k^*) r^\top(\gamma_k^*)))<\frac{4}{v_{dc}^{*2}}(Y^{-2}-1)$, with $d_k=(r_k^\top(\gamma_k^*) {J}^\top i_k^*)^{-2}+\frac{1}{4 K_p^2} \left(1+\eta C_{dc}  (\frac{1}{2} \mu r^\top(\gamma^*_k) {J}^\top i^*_{k})^{-1}\right)^2$.
	
Under $4/v_{dc}^{*2}\,(Y^{-2}-1)-\max\limits_{k=1,\dots, n} (r_k^\top(\gamma^*_k) {J}^\top i^*_{k})^{-2}> 0$, we solve for the gain $K_p$ with $\sigma_k=\overline\sigma(r(\gamma_k^*)r^\top(\gamma_k^*))=1$, to find
	
\begin{align*}
\sqrt{\frac{\max\limits_{k=1,\dots,n} \frac{\mu^2}{4} (1+\eta C_{dc}(\frac{1}{2}  \mu\, r^\top(\gamma^*_k)\, {J}^\top i^*_{k})^{-1})^2} {\frac{4}{v_{dc}^{*2}}(Y^{-2}-1) -\max\limits_{k=1,\dots, n} (r^\top(\gamma^*_k)\, {J}^\top i^*_{k})^{-2}}}< K_p.
\end{align*}
This can be simplified into \eqref{eq: dc-side}. The condition $4/v_{dc}^{*2}\,(Y^{-2}-1)-\max\limits_{k=1,\dots, n} (r_k^\top(\gamma^*_k) {J}^\top i^*_{k})^{-2}> 0$ can be written as $ Q^{\tb2}_{x,k}> \frac{\mu^2 v^{*4}_{dc}}{16 (Y^{-2}-1)}$,
 under the condition that $Y<1$ and we deduce that,
 {\tb  \begin{align*}
 	\max \left\{\frac{\mu^{2} v_{dc}^{*2}}{16\, R}, \tb\frac{\mu\, v_{dc}^{*2}}{4\, \sqrt{Y^{-2}-1}}\right\}<Q_{x,k}.  
 	\end{align*}
From the definition of the power factor $\cos(\phi_k)=\frac{P_{x,k}}{\sqrt{Q_{x,k}^2+P_{x,k}^2}}$, we arrive at \eqref{eq: ac-side}. 
In summary, we arrive at the sufficient conditions \eqref{eq: dc-side} and \eqref{eq: ac-side}.}
By applying Theorem \ref{thm: las}, we deduce that $\text{span}\{v(x^*)\}$ is asymptotically stable for the linearized system~\eqref{eq: lin-sys}.
	 \end{proof}

 \subsection{Results contextualization}
\label{subsec: contextualizing}
{\tb In what follows, we discuss Assumption~\ref{ass:linear-stab}.}
Generally speaking, Condition \eqref{eq: ac-side} can be regarded as a condition on the AC side, whereas \eqref{eq: dc-side} is a condition on the DC side control.
{\tb Both conditions are sufficient for stability and can be evaluated in a centralized fashion.}

{\tb Condition \eqref{eq: ac-side} connects the efficiency of the converter given by the power factor that defines the amount of current producing useful work 
to the lower bound $\alpha>0$. From \eqref{eq: ac-side}, the power factor may approaches 1, as $\alpha\to 0$.
}
If $\max \left\{\frac{\mu^{2} v_{dc}^{*2}}{16\, R}, \tb\frac{\mu v_{dc}^{*2}}{4\, \sqrt{Y^{-2}-1}}\right\}=\frac{\mu^{2} v_{dc}^{*2}}{16\, R}$, then condition \eqref{eq: ac-side} depends on the converter's resistance $R$, modulation amplitude $\mu$, nominal DC voltage $v_{dc}^*$, and the steady state current $i^*$. This is a known practical stability condition \cite{wang2014virtual}. 
In fact from \eqref{eq: ac-side},  sufficient resistive damping is often enforced by {\it virtual impedance control} which makes $\alpha\to 0$. 

If $\max \left\{\frac{\mu^{2} v_{dc}^{*2}}{16\, R}, \tb\frac{\mu v_{dc}^{*2}}{4\, \sqrt{Y^{-2}-1}}\right\}=\tb\frac{\mu v_{dc}^{*2}}{4\, \sqrt{Y^{-2}-1}}$, then we can again deploy $\mathcal{H}_\infty$ control to make {\tb $\norm{G_{ac}}_\infty$ arbitrarily small and thus $\alpha\to 0$}.
We note that, the AC side feedback control is crucial to achieve desired steady states for our power system model \eqref{eq: non-lin}. This can be implemented e.g., via outer loops that take measurements from the AC side and using the classical vector control architecture for the regulation of the inductance current and the output capacitor voltage \cite{d2015virtual}. 	
 
The condition $Y<1$ translates into the requirement that, $$\norm{G_{ac}}_\infty<\beta,\; \beta= \frac{2\,L}{\mu v^*_{dc}},$$ where $G_{ac}(j\tb\zeta)=(j{\tb\zeta} \mathbf{I}-F)^{-1}$ asks for $\mathcal{L}_2$ gain from the disturbances on the AC side to AC signals to be less than $\beta$. This can be achieved via $\mathcal{H}_\infty$ control \cite{zhou1996robust}.

Condition \eqref{eq: dc-side} depends on the steady state angles $\gamma^*$ and the converter and network parameters and asks for damping as the case for other stability conditions obtained in the literature on the study of synchronous machines \cite{arghir2018grid, dorfler2012synchronization}. The smaller is the {\em synchronization} gain $\eta>0$, the larger is the operating range of the DC damping gain $\widehat K_p$. 

{\tb For a more general setting with heterogeneous converters and  transmission lines parameters, our stability analysis can be  applied and analogous sufficient conditions to  \eqref{eq: ac-side} and \eqref{eq: dc-side} can be derived. 
}
}

\section{Simulations}
\label{sec: sims}

\begin{figure}[ht!]
	\centering
	\includegraphics[scale=0.27, trim= 1.5cm 0 0 0 ]{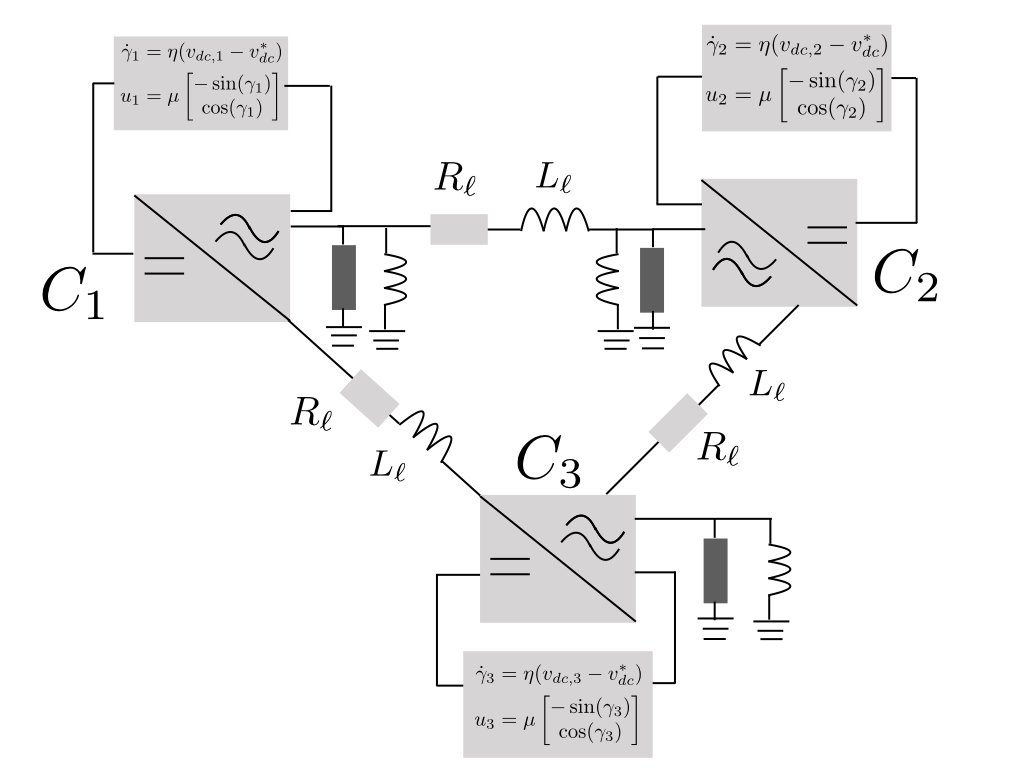}
	\caption{Three converter setup with system dynamics {\tb described by  \eqref{eq: multi-node}}, composed of identical three-phase converters $C_1, C_2$ and $C_3$ in closed-loop with the matching control and interconnected via identical RL lines. \tb The internal dynamics of the converters are modeled according to Figure \ref{fig:circuit diagram}.}
	\label{fig:my_label}
\end{figure}

{\tb The goal of this section is to  assess the asymptotic {\tb stability} of the trajectories of  the nonlinear power system \eqref{eq: non-lin} in Theorem \ref{thm: las} locally, i.e., by numerically estimating the region of attraction $\mathcal{D}$ in the neighborhood of $\mathcal{S}(z^*)$ for ${\tb z^*=[\gamma^{*\top},0^\top, v^{*\top}_c,i^{*\top}_\ell ]^\top}$}.

Let us consider three identical DC/AC converter model in closed-loop with the matching control depicted in Figure \ref{fig:my_label} and connected via three identical RL lines, as in \eqref{eq: non-lin} and connected to an inductive  and resistive load. Table \ref{table_example} shows the converter parameters and their controls  {\tb (in S.I.)}.

First, we  start by verifying the parametric conditions established in Assumption \ref{ass:linear-stab} via \eqref{eq: ac-side} and \eqref{eq: dc-side}. We tune the filter resistance $R>0$ (e.g. using virtual impedance control) so that \eqref{eq: ac-side} is satisfied. Next we choose the DC side gain $K_p>0$ so that \eqref{eq: dc-side} is satisfied.

Second, we numerically estimate the region of {\tb attraction} $\mathcal{D}$ of $\mathcal{S}(z^*)$ in the angle or $\gamma-$ space  by initializing sample trajectories of the angles  depicted in Figure \ref{fig:my_label}  at various locations  and illustrate the evolution of the nonlinear angle trajectories of \eqref{eq: non-lin} to estimate the region of attraction $\mathcal{D}$.
As predicted by Theorem \ref{thm: las}, we observe that the set $\mathcal{S}(z^*)$ restricted to the angles (relative to their steady state) space, and represented by $\textrm{span}\{\mathds{1}_3\}$ is asymptotically stable for the sampled angle trajectories of \eqref{eq: non-lin}.

{\tb Figure~\ref{fig: roc} depicts a projection onto the relative $(\gamma_1, \gamma_2, \gamma_3)-$ space of the estimate  of $\mathcal{D}$ (in rad). The convergence of angle solutions to the subspace $\mathds{1}_3$ is guaranteed for initial conditions at distance of $3.1$ resulting from varying the initial angles, while keeping the remaining initial states fixed. In particular, DC voltages and AC currents are also initialized close to their steady state values, as shown in Figure \ref{fig: voltages}. Our simulations show that the DC capacitor voltage $v_{dc}$ in Figures \ref{fig: voltages} and the AC output capacitor voltage in $abc-$ frame, namely $v$ converge to a corresponding steady state. This validates our theoretical results from Section \ref{sec: local-asymptotic stability}.
}

\begin{figure}[h!]
	\centering
	\includegraphics[scale=0.45]{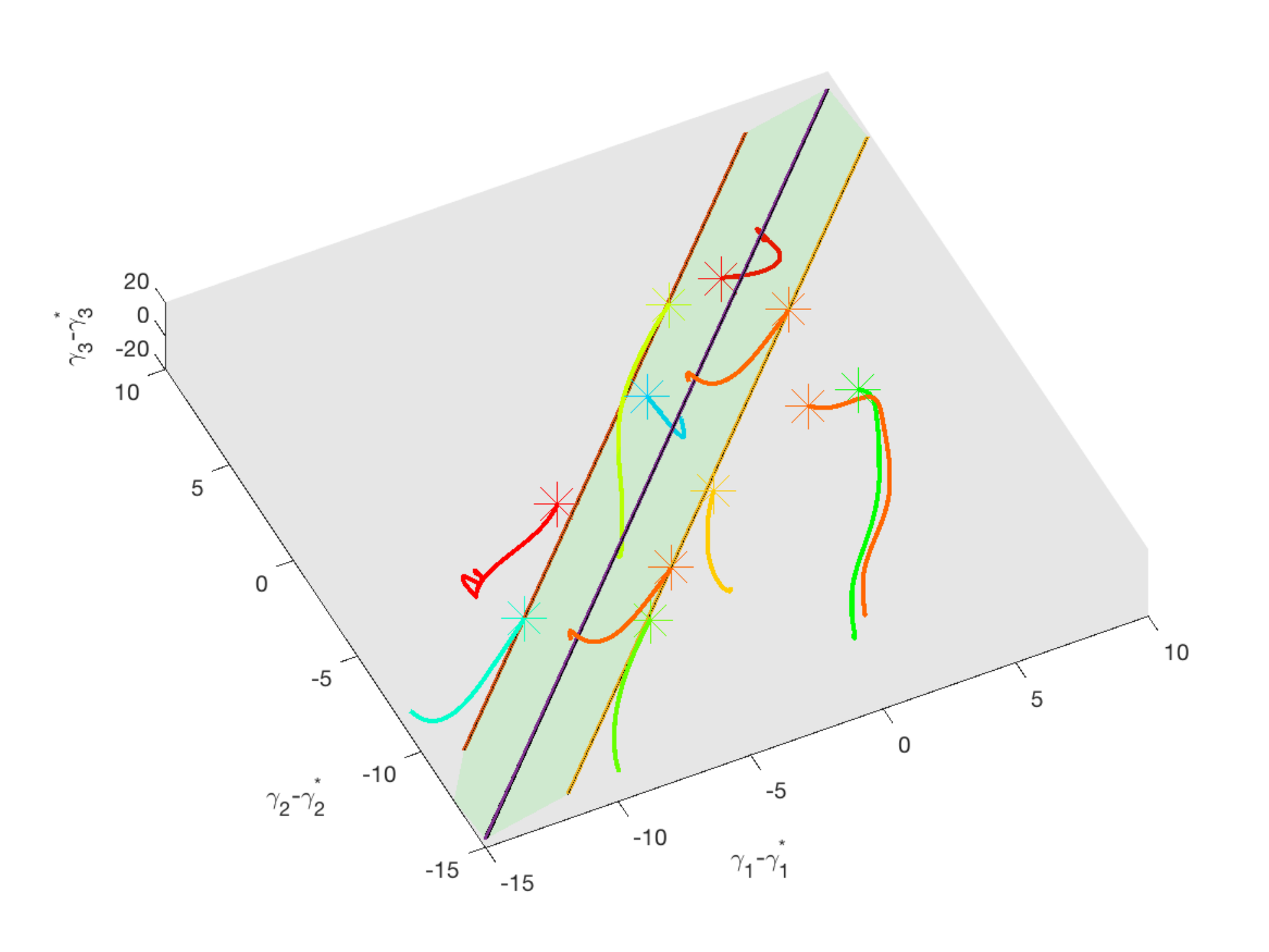}
	\caption{A plot of the region of {\tb attraction} $\mathcal{D}$  and the steady state set $\mathcal{S}(z^*)$ restricted to $( \gamma_1-\gamma_1^*,\gamma_2-\gamma_2^*,\gamma_3-\gamma_3^*)-$ space of the three DC/AC converter angles and convergence of the sample angle trajectories of \eqref{eq: non-lin} to the subspace $\mathds{1}_3$ within a distance of $d=3.1$. This results from varying the initial angles, while keeping the remaining initial states fixed. A sample of angles deviations initialized within the green area and denoted by different stars converge towards the stable set, while some angle trajectories initialized outside the estimated region are divergent. All the angles are represented in rad.}
	\label{fig: roc}
\end{figure}

\begin{figure}[h!]
	\centering
	\includegraphics[scale=0.5]{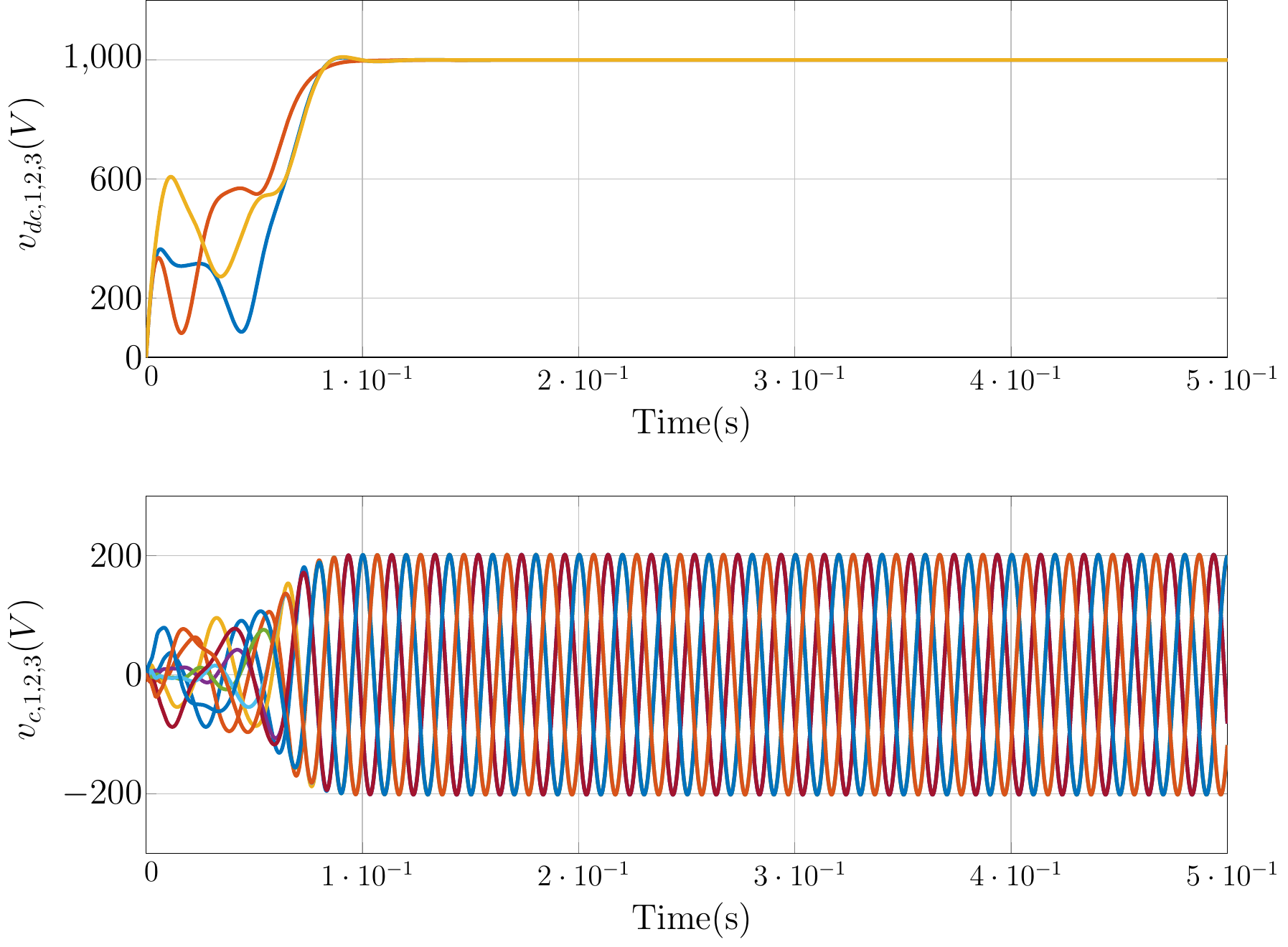}
	\caption{\tb Synchronization of DC capacitor voltages corresponds to frequency synchronization at the desired value. Hereby the angles are initialized at $(-6,-2,-13.15)$ (in rad) and belong to the projected  region of attraction shown in green in Figure \ref{fig: roc}.  The output capacitor voltage $v_c$ in $(abc)$ frame converges to a sinusoidal steady state $v^*_c$ as shown in Figure~\ref{fig: roc}.}
	\label{fig: voltages}
\end{figure}


{\tb For completeness, we have also illustrated a projection of the level sets of the Lyapunov function given in \eqref{eq:LF1}  of an example network consisting of {\em two} DC/AC converters interconnected via RL line that have the same dynamics as in \eqref{eq: non-lin} into $(\gamma_{1}-\gamma_1^*, \gamma_2-\gamma_2^*)-$ space in Figure \ref{fig: LF_sets}. The parameter values can be taken from Table \ref{table_example}. We set}
$
v(z^*)=\begin{bmatrix}
v_1^{\top}(z^*) & v_2^{\top}(z^*)
\end{bmatrix}^{\top}\in\ker(A(z^*))$, 
$v_1(z^*)=\left[\begin{smallmatrix}
0.043, & 0.043, & 0, & 0 
\end{smallmatrix}\right]^{\top}$,
and
$v_2(z^*)=\left[\begin{smallmatrix}
-0.0033&
-0.0023&
-0.0033&
-0.0023&
-0.7034&
-0.0108&
-0.7034&
-0.0108&
0 &
0
\end{smallmatrix}\right]^{\top}$
for a corresponding matrix $P>0$ as defined in \eqref{eq:matrix-P}. The function $V(x)$ takes positive values everywhere and is zero on the subspace spanned by $v(x^*)$.

\begin{figure}[ht!]
\centering{
\includegraphics[scale=0.51, trim= 1.5cm 5cm 0 3.5cm]{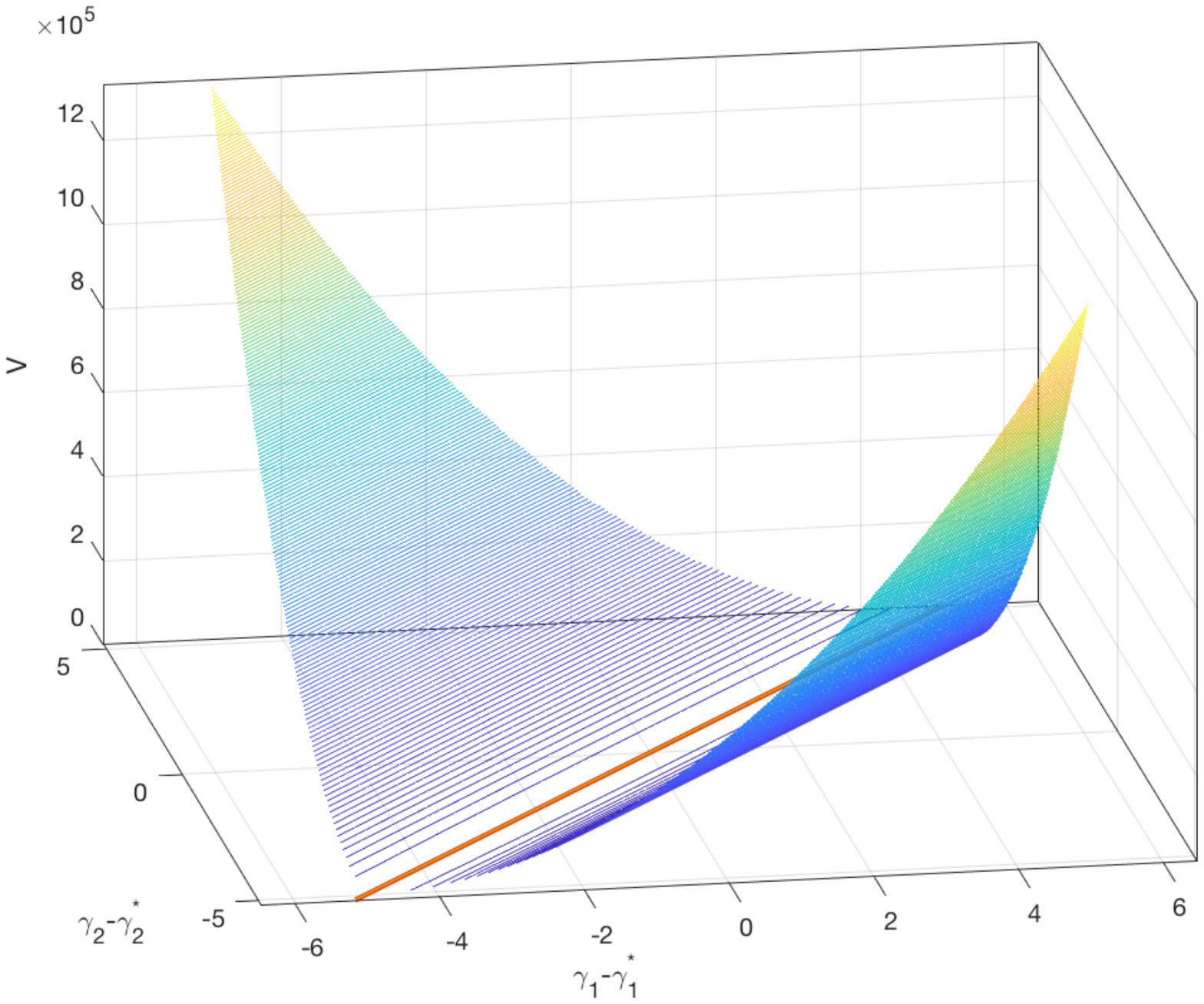}
\caption{\tb 3D representation of the Lyapunov function $V(x)$ in \eqref{eq:LF1} for {\em two} DC/AC converters in closed-loop with the matching control and connected via an RL line in \eqref{eq: multi-node} after a projection into $(\gamma_1-\gamma_{1}^*,\gamma_2-\gamma_2^*)$ space for $P>0$ as in \eqref{eq:matrix-P} and the subspace spanning $v(x^*)$. The parameter values can be found in Table~\ref{table_example}.}
\label{fig: LF_sets}}
\end{figure}

\begin{table}[h!]
	\begin{center}
		\begin{tabular}{|l||l|l|}
			\hline
			& $C_i,\; i=\{1,2,3\}$ & RL Lines\\
			\hline\hline
			$i_{dc}^*$  & $ 16.5$ & --\\
			$v_{dc}^*$ & $1000$& -- \\
			$C_{dc}$ & $10^{-3}$& --\\
			$G_{dc}$ & $10^{-5} $ & --\\
			$ K_P$& 0.099 &-- \\
			$\eta$ & 0.0003142 & -- \\
			$\mu$ & 0.33 &-- \\
			$L$ & $5\cdot10^{-4}$ &-- \\
			$C$ & $10^{-5}$ &-- \\
			$G$ & 0.1 & -- \\ 
			$R_{}$ & 0.2 &-- \\
			$R_{ \ell}$ &--  & $0.2$ \\
			$L_{ \ell}$ & -- & $5\cdot 10^{-5}$\\
			\hline
		\end{tabular}
	\end{center}
	\caption{Parameter values of DC/AC converters and the RL lines (in S.I).}
	\label{table_example}
\end{table}

\section{CONCLUSIONS}
We investigated the characteristics of a high-order steady state manifold of a multi-converter power system, by exploiting the symmetry of the vector field.  We studied local {\tb asymptotic stability of} the steady state set {\tb as a direct application of the center manifold theory} and provided an operating range for the control gains and parameters. Future directions include finding better estimates of the region of {\tb attraction} using advanced numerical methods and more {\tb detailed} simulations of high-order power system models.

{\section*{ACKNOWLEDGMENT}
The authors would like to kindly thank Florian Dörfler, Anders Rantzer, Richard Pates and Mohammed Deghat for the insightful and important discussions.}




\bibliographystyle{IEEEtran}
\bibliography{bib/root}

\end{document}